\documentclass[onefignum,onetabnum]{siamart171218}

\usepackage{lipsum}
\usepackage{amsfonts}
\usepackage{graphicx}
\usepackage{epstopdf}
\usepackage{algorithmic}
\usepackage{amsmath,amssymb,mathrsfs}
\usepackage{color}
\usepackage{latexsym}
\usepackage{bm,bbm}

\ifpdf
  \DeclareGraphicsExtensions{.eps,.pdf,.png,.jpg}
\else
  \DeclareGraphicsExtensions{.eps}
\fi

\newcommand{\Prob}[1]{\mathbb{P}\left[#1\right]}

\newcommand{\Exp}[1]{\mathbb{E}\left[#1\right]}
\newcommand{\ExpPN}[1]{\mathbb{E}^0_{\bm N}\left[#1\right]}
\newcommand{\ExpN}[1]{\mathbb{E}_{\bm N}\left[#1\right]}

\newcommand{\Expi}[1]{\mathbb{E}^0_i\left[#1\right]}
\newcommand{\Expj}[1]{\mathbb{E}^0_j\left[#1\right]}
\newcommand{\Expk}[1]{\mathbb{E}^0_k\left[#1\right]}
\newcommand{\Expij}[1]{\mathbb{E}^0_{ij}\left[#1\right]}

\newsiamremark{remark}{Remark}
\newsiamremark{hypothesis}{Hypothesis}
\crefname{hypothesis}{Hypothesis}{Hypotheses}
\newsiamthm{claim}{Claim}

\headers{RMF neural networks}{F. Baccelli and T. Taillefumier}

\title{The pair-replica-mean-field limit \\ for intensity-based neural networks}

\author{Fran\c{c}ois Baccelli \thanks{Department of Mathematics and Department of Electrical and Computer Engineering, University of Texas, Austin, TX 
  (\email{francois.baccelli@austin.utexas.edu}).}
\and Thibaud Taillefumier \thanks{Department of Mathematics and Department of Neuroscience, University of Texas, Austin, TX 
  (\email{ttaillef@austin.utexas.edu}).}
  }

\usepackage{amsopn}

\ifpdf
\hypersetup{
  pdftitle={The Pair-Replica-Mean-Field Limit for Intensity-based Neural Networks},
  pdfauthor={F. Baccelli and T. O. Taillefumier}
}
\fi

\begin{document}

\maketitle

\begin{abstract}
Replica-mean-field models have been proposed to decipher the activity of neural networks
via a multiply-and-conquer approach. 
In this approach, one considers limit networks made of infinitely many replicas with
the same basic neural structure as that of the network of interest, but
exchanging spikes in a randomized manner. 
The key point is that these replica-mean-field networks are tractable
versions that retain important features of the finite structure of interest. 
To date, the replica framework has been discussed for first-order models,
whereby elementary replica constituents are single neurons with independent Poisson inputs. 
Here, we extend this replica framework to allow elementary replica constituents to be
composite objects, namely, pairs of neurons. 
As they include pairwise interactions, these pair-replica models exhibit non-trivial
dependencies in their stationary dynamics, which cannot be captured by first-order replica models. 
Our contributions are two-fold: 
$(i)$ We analytically characterize the stationary dynamics of a pair of
intensity-based neurons with independent Poisson input. 
This analysis involves the reduction of a boundary-value problem related to a
two-dimensional transport equation to a system of Fredholm integral equations---a result of independent interest. 
$(ii)$ We analyze the set of consistency equations determining the full network dynamics
of certain replica limits. 
These limits are those for which replica constituents, be they single neurons or pairs of neurons,
form a partition of the network of interest. 
Both analyses are numerically validated by computing input/output transfer functions for neuronal pairs
and by computing the correlation structure of certain pair-dominated network dynamics.
\end{abstract}





\section{Introduction} The present work focuses on neural models which represent the spiking 
neuronal activity in terms of point processes \cite{Daley1,Daley2}. 
In these models, the rate of spiking of each neuron is governed by a ``stochastic intensity''
that integrates afferent neural inputs, thereby mediating network interactions. 
These intensity-based networks constitute a natural and flexible class of models for neural activity,
whose study has a long and successful history in neuroscience
\cite{Bialek:1999,DayanAbbott,Emery:2004,Pillow:2008aa}.
Unfortunately, detailed computational knowledge of intensity-based networks is mostly
limited to simplifying limits such as the thermodynamic limit, i.e., with a very large
number of neurons interacting very weakly \cite{Amari:1975aa,Amari:1977aa,Sompolinski:1988,Faugeras:2009}.  
Such limitations preclude explaining and predicting several key aspects of neural computations that
emerge from the finite size of neural components, including activity correlation \cite{Goris:2014aa,Lin:2015aa},
dynamical metastability \cite{Abeles:1995,Tognoli:2014aa},
and computational irreversibility \cite{Feynman:1998}.
Indeed, in the thermodynamic limit, activity correlation vanishes, dynamical metastability disappears,
and computations are all reversible.
There is a need for a computational framework allowing for an analysis of structured neural networks
that reproduces the key features of finite-size circuits.

\subsection{Background} In a recent work \cite{Baccelli:2019aa}, we introduced such
a computational framework, called a replica-mean-field (RMF) framework, by adopting
the {\em multiply-and-conquer} approach. In this approach, we consider limit networks,
the so-called RMF limits, made of infinitely many replicas with the same basic finite network structure
~\cite{Vvedenskaya:1996,Benaim:121369,Baccelli:2018aa}. 
Although physically sound, these RMF limits may look computationally intractable
for being infinite dimensional. However, simulation reveals that RMF limits exhibit
the property of asymptotic independence between replicas and feature Poisson input point processes.
This is referred to as the ``Poisson Hypothesis'' in network theory~\cite{RybShlosI}.
This property significantly simplifies the analysis of the network dynamics. 
In fact, despite the infinite number of replicas, RMF limits constitute tractable
neural networks that retain key features of the dynamics of interest. 
The inclusion of finite network structure in RMF dynamics promises a first characterization
of the aspects of neural activity mentioned above: correlations between neurons, 
dynamical metastability, and computational irreversibility. 
In our introductory work~\cite{Baccelli:2019aa}, we considered RMF limits for a class
of excitatory, intensity-based networks, called linear Galves-Loch\"erbach (LGL)
models~\cite{Galves:2013,DeMasi:2015}.
In LGL models, the stochastic intensities serve as neuronal state variables which integrate
impulse-like spike deliveries, while continuously relaxing to a base rate and instantaneously
reseting upon spiking. The RMF limits of these LGL models considered in~\cite{Baccelli:2019aa}
assume the most stringent Poisson Hypothesis, whereby neurons are independent encoders
that spike with self-consistently determined stationary input rates.
These stringent---or rather first-order---RMF limits were shown to retain some of the structural
properties of LGL networks such as their rate saturation in the limit of large synaptic weights;
they also yield the explicit dependence of the spiking rates on the size of neural constituents.
However, these computational feats come at the cost of erasing all correlation structures in the RMF limits.

\subsection{Aim and Contributions}\label{ssec:aim} 
The purpose of this work is to extend the RMF framework to include
correlations among neuronal pairs, thereby introducing pair-replica-mean-field (pair-RMF) limits. 
Such pair-RMF limits are obtained by considering that independently interacting replica constituents
can be neuronal pairs in addition to single neurons\footnote{In contrast, first-order RMF
limits only comprise single neurons as replica constituents.}.
The pair-RMF approach has two components: 
$(i)$ A characterization of the stationary state of replica constituents defined as pairs
of neurons subjected to independent Poisson bombardment from upstream neurons.
$(ii)$ The construction of a network model connecting these constituents via
a set of self-consistency input rate equations.
The equations intervening in $(ii)$ are precisely those obtained from solving $(i)$.
The explicit solutions of the PDEs obtained in $(i)$ are thus the basis of the analysis conducted in $(ii)$,
in addition of being of independent mathematical interest. However, the RMF analysis conducted in $(ii)$
is the most meaningful contribution to the understanding of the quantitative
and qualitative properties of large neural networks.
Another contribution in $(ii)$ is the proof of the existence of solutions
to the self-consistency equations as a corollary of the existence of RMF limits.

\subsection{Methodology} As explained above, our first 
contribution is the analysis of a connected pair of neurons receiving
independent Poissonian spiking deliveries from upstream neurons.
Let us first describe the methodology used for this analysis.
Just as for first-order RMF limits, our strategy is to characterize the neuronal pair's
stationary state via its moment-generating function (MGF).
The MGF of a single neuron with Poissonian inputs satisfies an ordinary differential equation (ODE)
parameterized by the rate of this input. The solution to this ODE can be found by imposing some
analyticity requirements that any MGF must satisfy. By contrast the analysis of a pair of
neurons with Poisson input involves a two-variable MGF, denoted by $(u,v) \mapsto L(u,v)$.
Crucially, this MGF is solution to a partial differential equation (PDE) instead of an ODE.
Thus, the main challenges to characterize stationary pair dynamics is to extend our
ODE analysis to the PDE setting.
The first challenge consists in finding a determinate form for the PDE solution provided
by the method of characteristics~\cite{Evans:2010}.
This can be done via a boundary analysis of the PDE, which reveals the key role played 
by the boundary fluxes $\partial_u L(0,v)$ and $\partial_v L(u,0)$.
The second challenge consists in exploiting the requirement of analyticity of the
PDE solution to deduce functional equations determining the neuronal pair stationary rate.
This can be done via elementary analytic manipulations to yield a system of two coupled
homogeneous Fredholm equations bearing on the boundary fluxes $\partial_u L(0,v)$ and
$\partial_v L(u,0)$, in addition to a normalization condition \cite{Polyanin:2008}.
The boundary fluxes $\partial_u L(0,v)$ and $\partial_v L(u,0)$ are closely related to
the MGFs of the neuronal Palm distributions, which characterize the typical state
of the neuronal pair at a spiking time~\cite{Mathes:1964,Mecke:1967}.
This relation allows us to give a probabilistic interpretation to the integral
equations via the rate-conservation principle applied to the embedded Markov
chain~\cite{BremaudBaccelli:2003} of the pair.
This probabilistic interpretation establishes the uniqueness of the solution to the PDE as a by-product.

Our second contribution is the determination of the consistency equations allowing for
the analysis of a large network in terms of its constituents. 
We primarily focus on the case where
these constituents are either pairs or single neurons obtained from a partition
of the set of all neurons of a large network.
One can then leverage the solutions of the ODEs and PDEs
associated to these constituents to consistently determine their stationary input rates in the RMF limit. 
The resulting consistency equations can be interpreted as Kirchhoff-type conservation laws
stipulating the balance between input activity and output activity for each constituent \cite{Feynman:1964vol2}.
These Kirchhoff-type laws exhibit a few original features:
First, these laws are nonlinear with respect to input rates due to the inclusion of
post-spiking reset rules---the exclusive source of nonlinearity in LGL dynamics~\cite{Galves:2013,DeMasi:2015}.
Second, these laws weight non-multiplicatively the role of input rates and synaptic weights
in shaping the input/output balance---a hallmark of finite-size effects in RMF approaches~\cite{Baccelli:2019aa}.
Third, these laws can bear on the input/output balance of composite objects,
namely pairs of neurons, with nontrivial internal correlations---a specificity of the pair-RMF approach
discussed in the present paper.
The establishment of such Kirchhoff-type conservation laws is ultimately made possible by the
Poissonian nature of the constituents' inputs in the RMF limit, which in turn follows from
the property of asymptotic independence between replicas in the RMF limit.
The rigorous justification of this asymptotic independence is the object of
a forthcoming paper~\cite{Davydov:2020}. 
RMF limits can be viewed as physical probabilistic systems whose stationary input
rates solve these consistency equations.
This provides a new computational framework to analyze large neural networks.

\subsection{Structure}  In \Cref{sec:RMFA}, we characterize the stationary regime
of a neuronal pair receiving independent Poissonian spike deliveries via a PDE,
called the pair-PDE, whose study is the first focus of this work.
For completeness, we include a derivation of the pair-PDE
by application of the rate-conservation principle of Palm calculus in Appendix.
In \Cref{sec:APPDE}, we present our first contribution, i.e., a solution to the pair-PDE, which we deduce by
imposing requirements of analyticity that any probabilistic solutions should meet.
We provide the probabilistic interpretation of this solution in Appendix. 
In \Cref{sec:transfer}, we utilize our analytical results to study
the input/output transfer function of a neuronal pair, which includes pair-correlation estimates. 
This section also contains a closed form expression for the solution to the PDE when
the neuronal pair receives no external inputs.
In \Cref{sec:RI}, we discuss our second contribution, i.e., the use of the
solution of the pair-PDE for RMF limits of large networks. As explained above,
these RMF limits are physical systems whose stationary states satisfy the
self-consistency rate equations.
We discuss these consistency equations for RMF limits associated to a partition of the network 
in neuronal pairs or singletons (this will be referred to
as the pair-partition RMF) and to its all-pair decomposition (all-pair-RMF).

\subsection{Related work} The inspiration for the replica models analyzed 
in this work is rooted in the theory of nonlinear Markov processes, 
which were introduced by McKean \cite{McKean:1966}. These processes
were extensively used to study mean-field limits in queueing systems, initially by the Dobrushin school 
\cite{Vvedenskaya:1996,Rybko:2009aa,RybShlosI,RybShlosII}, and later by M. Bramson \cite{bramson:2011}.
This literature has two distinct components: 
$a)$ a probabilistic component proving asymptotic independence from the equations satisfied by the
non-linear Markov process, and $b)$ a computational component deriving closed-form expressions
for the mean-field limit of the system of interest.
These two components jointly led to a wealth of results in queueing theory (see, e.g., \cite{Vvedenskaya:1996}). 
Our work in \cite{Baccelli:2019aa} showed that, just as in queueing theory, studying neural
networks in the RMF limit is computationally tractable, at least for first-order RMF models.
The purpose of this work is to extend the RMF framework introduced in \cite{Baccelli:2019aa}
to include pairwise interactions. To this end, we resort to analytical methods to specify
the MGFs associated to RMF limits where elementary constituents can be single neurons
or pairs of neurons. Finding MGFs by imposing conditions of analyticity on some solutions
is a classical approach in queueing theory \cite{Takacs:1962aa}, and in a PDE context ~\cite{Fayolle:2017}.
However, the method used in this work, which consists in a boundary analysis of the
general solution provided by the method of characteristics, is novel.

Our approach is part of a rich line of prior attempts to solve the neural master equations in
computational neuroscience. Brunel {\it et al.} introduced mean-field limits for large
neural networks with weak interactions from a statistical physics perspective
\cite{Abbott:1993,Brunel:1999aa, Brunel:2000aa}. Touboul {\it et al.} then adapted the
ideas of ``propagation of chaos'' for neural networks in the thermodynamic
mean-field limit \cite{Baladron:2012aa, Touboul:2012aa, Robert:2016}.
Their results were specialized to spiking models with memory resets by
Galves and Locherb\"ach, who also provided perfect algorithms to simulate the
stationary states of infinite networks \cite{Galves:2013,DeMasi:2015}.
In a more computational approach, Toyoizumi {\it et al.} determined finite-size
corrections to mean-field models of weakly interacting Hawkes processes \cite{Toyoizumi:2008aa}.
Schwalger {\it et al.} captured finite-size effects in large but finite intensity-based 
neural networks by developing a quasi-renewal approximation for mesoscopic
neural populations  \cite{Schwalger:2017aa}. 
Dumont {\it et al.} analyzed similar finite-size effects in intensity-based neural
networks by studying stochastic partial differential equations obtained via a linear
Gaussian approximation \cite{Dumont:2017aa}.
Closer to our approach, Buice, Cowan, and Chow adapted techniques from statistical physics to analyze
the hierarchy of moment equations obtained from the master equations \cite{Buice:2007,Buice:2009aa}.
These authors were able to truncate the hierarchy of moment equations to consider models
amenable to finite-size analysis via system-size or loop expansion around the mean-field
solution \cite{Bressloff:2009aa}.


\section{The stationary regime of a neuronal pair}
\label{sec:RMFA}


In this section, we characterize the elementary dynamics at the crux of the present work. 
In \Cref{sec:LGL}, we introduce the so-called linear Galves-L\"ocherbach (LGL) neuronal
dynamics of general dimensions.
In \Cref{ssec:RCE}, we derive from the PDE satisfied by the MGF of general LGL networks the PDE associated to a pair of LGL neurons receiving
independent Poissonian spike deliveries. 


\subsection{Linear Galves-L\"ocherbach models}\label{sec:LGL}

We consider a finite assembly of $K$ neurons whose spiking activity is modeled as the realization
of a system of simple point processes without common points $\bm{N} = \lbrace N_i \rbrace_{1\leq i \leq K}$
on $\mathbb{R}$ defined on an underlying measurable space $(\Omega, \mathcal{F})$.
For all neurons $1\leq i \leq K$, we denote by $\lbrace T_{i,n} \rbrace_{n \in \mathbb{Z}}$,
the sequence of successive spiking times with the convention that almost surely
$T_{i,0} \leq 0 < T_{i,1}$ and $T_{i,n}<T_{i,n+1}$ (this is the customary numbering convention
for stationary point processes).
Each point process $N_i$ is a family $\lbrace N_i(B)\rbrace_{B \in \mathcal{B}(\mathbb{R})}$
of random variables with values in $\mathbb{N} \cup \lbrace \infty \rbrace$ 
indexed by the Borel $\sigma$-algebra $\mathcal{B}(\mathbb{R})$ of the real line $\mathbb{R}$.
Concretely, the random variable $N_i(B)$ counts the number of times that neuron $i$ spikes within 
the set $B$, i.e., $N_i(B) = \sum_{n \in \mathbb{Z}} \mathbbm{1}_B(T_{i,n})$. 
Setting the processes $N_i$, $1 \leq i \leq K$, to be independent Poisson processes
defines the simplest instance of our point-process framework as a collection of non-interacting neurons.

To model spike-triggered interactions within the network, we consider that the rate
of occurrences of future spikes depends on the spiking history of the network.
In other words, we allow the instantaneous spiking rate of neuron $i$ to depend on the
times at which neuron $i$ and other neurons $j \neq i$ have spiked in the past.
Formally, the network spiking history $\lbrace \mathcal{F}_t \rbrace_{t \in \mathbb{R}}$
is defined as a non-decreasing family of $\sigma$-fields such that, for all $t$,
\begin{eqnarray}
 \mathcal{F}_t^{\bm{N}} = \left\{ \sigma\left( N_1(B_1), \ldots, N_K(B_K)\right) \, \vert \, B_i \in \mathcal{B}(\mathbb{R}) \, , \: B_i \subset (-\infty,t] \right\} \subset \mathcal{F}_t ,
\end{eqnarray}
where $\mathcal{F}_t^{\bm{N}}$ is the internal history of the spiking process $\bm{N}$.
The network spiking history $\lbrace \mathcal{F}_t \rbrace_{t \in \mathbb{R}}$ determines
the rate of occurrence of future spikes via the notion of stochastic intensity.
The stochastic intensity of neuron $i$, denoted by $\lbrace \lambda_i(t) \rbrace_{t \in \mathcal{R}}$,
can be seen as a function of $\lbrace \mathcal{F}_t \rbrace_{t \in \mathbb{R}}$
specifying the instantaneous spiking rate of neuron $i$.
It is formally defined as the $\mathcal{F}_t$-predictable process
$\lbrace \lambda_i(t) \rbrace_{t \in \mathbb{R}}$ satisfying 
\begin{eqnarray*}
\Exp{N_i(s,t] \, \vert \, \mathcal{F}_{s}}= \Exp{\int_s^t \lambda_i(u) \, du\, \Big \vert \, \mathcal{F}_{s}} \, ,
\end{eqnarray*}
for all intervals $(s,t]$ \cite{Jacod:1975}. 
Stochastic intensities generalize the notion of rate of events, or hazard function,
to account for generic history dependence beyond that of Poisson processes or renewal processes.

Specifying the history-dependence of the neuronal stochastic intensities entirely defines
a network model within the point-process framework. In this work, we consider models for which
the stochastic intensities $\lambda_1, \ldots, \lambda_K$ obey the following
system of coupled stochastic integral equations
\begin{eqnarray}\label{eq:dynmod}
\lefteqn{
\lambda_i(t) = \lambda_i(0) + \frac{1}{\tau_i} \int_0^t \left(b_i-\lambda_i(s) \right)\, ds +} \nonumber\\
&& \hspace{60pt} \sum_{j \neq i} \mu_{ij} \int_0^t N_j(ds) + \int_0^t \big(r_i-\lambda_i(s) \big) N_i(ds) \, ,
\end{eqnarray}
where the spiking processes $N_i$ have stochastic intensity $\lambda_i$. 
The above system of stochastic equations characterizes the history-dependence of the stochastic intensities.
The first integral term indicates that in between spiking events, $\lambda_i$ deterministically relaxes
toward its base rate $b_i>0$ with relaxation time $\tau_i$.
The second integral terms indicates that a spike from neuron $j \neq i$
causes $\lambda_i$ to jump by $\mu_{ij} \geq 0$, the strength of the synapse from $j$ to $i$.
Finally, the third integral term indicates that $\lambda_i$ resets to
$0 \leq r_i \leq b_i$ upon spiking of neuron $i$. 
Taking $r_i < b_i$ models the refractory behavior of neurons whereby spike generation causes
the neuron to enter a transient quiescent phase. 

In summary, the system parameters and the dynamics of this model are as follows:
$(i)$ The network has $K>0$ neurons connected via synaptic weights $\mu_{ij}$, $i\ne j =1,\ldots,K$.
$(ii)$ The state variables are the stochastic intensities ${\bm \lambda}(t) = (\lambda_i(t))_{i=1,\dots,K}$.
$(iii)$ In between spiking interactions, the intensity $\lambda_i(t)$ relaxes
toward the base rate $b_i>0$, with relaxation time $\tau_i>0$.
$(iv)$ When neuron $i$ spikes, $\lambda_i$ resets to $r_i$ with
$0 \leq r_i \leq b_i$ and for all $j\ne i$, $\lambda_j$ increases by $\mu_{ji}\ge 0$.
Thus-defined, our model can be seen as a system of coupled Hawkes processes with spike-triggered
memory reset and belongs to the Galves-L\"ocherbach class of models \cite{Galves:2013}.


\subsection{Partial differential equation for the pair dynamics}\label{ssec:pairPDE}

Consider a LGL network with $K$ neurons $i$, $1 \leq i \leq K$, specified by the relaxation times $\tau_i$,
the base rates $b_i$, the reset values $r_i$, $1 \leq i \leq K$, and the interaction weights
$\mu_{ij}$, $1 \leq i\ne j \leq K$.
In \cite{Baccelli:2019aa}, we showed that such a network defines a Harris ergodic Markov
chain with state variables ${\bm \lambda}(t) = ( \lambda_i(t) )_{i=1,\dots,K}$,
where $\lambda_i(t)$ denotes the stochastic intensity of neuron $i$ at time $t$.
Moreover, the dynamics of ${\bm \lambda}(t)$ converges at least exponentially fast
toward a stationary dynamics with exponential moments.
This allows one to characterize the stationary distribution of the network state via the stationary MGF
\begin{eqnarray}
L(u_1, \ldots, u_k) = \Exp{ \exp{\left( \sum_{i=1}^K u_i \lambda_i(0) \right) }} \, , \quad \mathrm{with} \quad u_1, \ldots, u_K \in \mathbbm{R} \, ,
\end{eqnarray}
where, by convention,  $\lambda_i(0)$ denotes the stationary stochastic intensity of neuron $i$.
In principle, one could characterize the above MGF as the solution of a  first-order linear PDE which was
given in \cite{Baccelli:2019aa}:

\begin{definition}\label{th:exactPDE}
The full $K$-dimensional MGF $L$ satisfies the PDE
\begin{eqnarray}\label{eq:exactPDE}
\left( \sum_i \frac{u_i b_i}{\tau_i}  \right) L - \sum_i \left( 1+\frac{u_i}{\tau_i} \right) \partial_{u_i} L
+
\sum_i e^{\left(u_i r_i + \sum_{j \neq i} u_j \mu_{ji}\right)} \partial_{u_i} L \Big \vert_{u_i=0} = 0 \, .
\end{eqnarray}
\end{definition}

\begin{figure}
  \centering
  \includegraphics[width=0.8\textwidth]{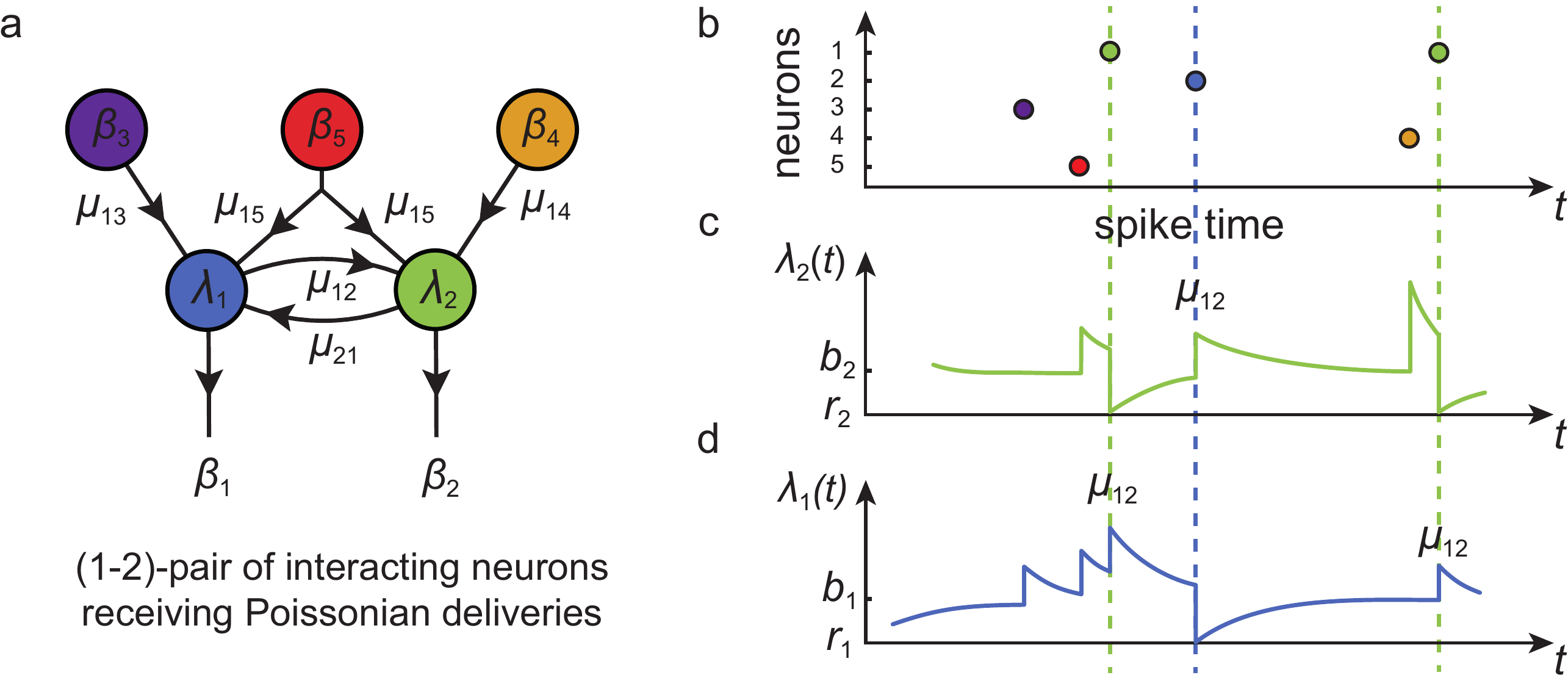}
  \caption{{\bf Elementary dynamics of a neuronal pair.}
   {\bf a.} Schematic of an interacting pair of neurons receiving independent Poissonian input from three upstream neurons.
  {\bf b.} The spiking times of the neurons define a point process determining the stochastic intensities of the neuronal pair.
  {\bf c.} and
  {\bf d.} Evolution of the stochastic intensities $(\lambda_1,\lambda_2)$ of the neuronal pair $(1,2)$ with relaxation toward base  rate $(b_1,b_2)$, post-spiking reset $(r_1,r_2)$, and interaction weights $\mu_{12}$ and $\mu_{21}$.
  }
  \label{fig:pair}
\end{figure}

In practice, the PDE of Theorem \ref{th:exactPDE} is analytically intractable
due to the presence of non-local boundary flux terms, except for some elementary dynamics.
These elementary dynamics, which are the focus of this section, are that of a neuronal pair
subjected to independent Poissonian spike deliveries (see \Cref{fig:pair}).
Compared to the full-dimensional picture, this corresponds to singling out a pair of
neurons $(i,j)$ receiving inputs from a given set of upstream neurons, $k \neq i,j$.
These upstream neurons provide independent Poissonian external drives
with intensities $\beta_k$.
Then, the stationary dynamics of the pair state
$(\lambda_i,\lambda_j)$ solves the system of stochastic equations
\begin{eqnarray}\label{eq:dynmod2a}
\lambda_i(t) &=& \lambda_i(0) + \frac{1}{\tau_i} \int_0^t \left(b_i-\lambda_i(s) \right)\, ds + \int_0^t \big(r_i-\lambda_i(s) \big) N_i(ds) \nonumber\\
&&  \mu_{ij} \int_0^t N_j(ds) + \sum_{k \neq i,j} \mu_{ik} \int_0^t N_k(ds) \, , \\
\label{eq:dynmod2b}
\lambda_j(t) &=& \lambda_j(0) + \frac{1}{\tau_j} \int_0^t \left(b_j-\lambda_j(s) \right)\, ds + \int_0^t \big(r_j-\lambda_j(s) \big) N_j(ds) \nonumber\\
&&   \mu_{ji} \int_0^t N_i(ds) + \sum_{k \neq i,j} \mu_{jk} \int_0^t N_k(ds) \, , 
\end{eqnarray}
where $N_i$ and $N_j$ have stochastic intensities $\lambda_i$ and $\lambda_j$, respectively,
and where the processes $N_k$, $k \neq i,j$, are independent stationary Poisson processes
with rate $\beta_k$. Observe that in both \eqref{eq:dynmod2a} and \eqref{eq:dynmod2b},
the first line specifies the autonomous evolution of neurons $i$ and $j$, whereas the second 
line collects the terms arising from pair interactions and from external drives.
With independent Poissonian drives, the joint stationary distribution of the pair
states $(\lambda_i,\lambda_j)$ is characterized by its two-variable MGF
\begin{eqnarray}
L_{ij}(u,v)=\Exp{e^{u \lambda_i(0)+v \lambda_j(0)}}=L\vert_{u_i=u,u_j=v,u_k=0, k \neq i, j},\quad 1\le i < j\le K \, ,
\end{eqnarray} 
which is defined for all real numbers $u, v$.
Observe that the mean spiking rates of the neuronal pair $(i,j)$, denoted by $(\beta_i,\beta_j)$, satisfy:
\begin{eqnarray}\label{eq:defBeta}
\beta_i=\Exp{\lambda_i}=\frac{\partial L_{ij}}{\partial u}L(0,0) \quad \mathrm{and} \quad \beta_j=\Exp{\lambda_j}=\frac{\partial L_{ij}}{\partial v}L(0,0) \, .
\end{eqnarray}
In the following, when discussing pair-related quantities, we will drop the subscripts
on $i$ and $j$ whenever possible. For instance, we will refer to $L_{ij}$ as $L$ when there is no ambiguity. 
With that in mind, our first goal is to specify the MGF $L$ as the solution of a PDE~\cite{Evans:2010}.
This PDE, which is given in the next definition, will be referred to as the pair-PDE of pair
$(i,j)$ in the initial network. 
This PDE can be obtained by specializing the full PDE \eqref{eq:exactPDE}
for the two-variable MGF $L$ by setting $u_k=0$ for $k \neq i,j$ and for the network structure of interest, i.e., by setting $\mu_{ki}=\mu_{kj}=0$ for all $k \neq i,j$.

\begin{definition}\label{def:pairPDE}
When subjected to independent Poissonian spike deliveries with rates $\beta_k$, $k\neq i,j$, 
the two-variable stationary MGF $L$ of the neuronal pair $(i,j)$ satisfies the pair PDE:
\begin{eqnarray}\label{eq:pairPDE}
\lefteqn{ \left( 1\!+ \! \frac{u}{\tau_i}\right)  \partial_u L+ \left( 1\!+ \! \frac{v}{\tau_j} \right) \partial_v L -\left( \frac{ub_i}{\tau_i}  +  \frac{vb_j}{\tau_j}  +\sum_{k \neq i,j}\left( e^{u \mu_{ik}+v \mu_{jk}}\!-\!1\right) \beta_k \right) L = } \nonumber\\
&&  \hspace{200pt} e^{u r_i+v \mu_{ji}}  \partial_u L \vert_{u=0}+  e^{v r_j+u \mu_{ij}}\partial_v L \vert_{v=0}  \, .
\end{eqnarray}
\end{definition}

The pair-PDE \eqref{eq:pairPDE} is a nonlocal first-order PDE with boundary terms
involving partial derivatives. Conceptually, it depicts the stationary state of a $2$-dimensional
transport equation in the region $[-\tau_i,0] \times [-\tau_j,0]$ of the $(u,v)$-plane, 
with linear drift $( 1+u/\tau_i, 1+v/\tau_j)$, 
with non-linear death rate involving the parameters $\beta_k$,
and with non-local birth rates related to fluxes through the hyperplane $\{ u = 0 \}$ and $\{ v = 0 \}$.
Despite this conceptual simplicity, the presence of flux-related, non-local, birth rates and nonlinear
death rates precludes one from solving \eqref{eq:pairPDE} explicitly, except for the simplest cases.
As we shall see, explicit solutions for $L$ are only possible at this stage in the absence of external
inputs, i.e., when $\beta_k=0$, $k \neq i,j$ (see \Cref{sec:nodrive}). 
A probabilistic proof of the pair-PDE is given in  \Cref{ssec:RCE}.


\section{Analysis of the pair-PDE}
\label{sec:APPDE}


In this section, we show that techniques from the MGF formalism allow one to reduce
the pair-PDE  \eqref{eq:pairPDE} to a set of integral equations bearing on the boundary terms of the pair-PDE.
In  \Cref{sec:char}, we give an integral representation for the solutions of the pair-PDE in
terms of the boundary terms and of an undetermined function.
In  \Cref{sec:bound}, we give a fully determined integral representation of the boundary terms.
In  \Cref{sec:simple}, we exploit the requirement of boundedness for the solution to lift 
the indeterminacy of the integral representation for solutions of the pair-PDE.
In  \Cref{sec:intSys}, we derive from the integral representation for the bounded solution an integral system of Fredholm equations characterizing the full solution of the pair-PDE.


\subsection{Integral representation via the method of characteristics}\label{sec:char}

The main hindrance to solving the pair-PDE \eqref{eq:pairPDE} is due to the presence of boundary terms.
Here, we temporarily sidestep this hindrance by assuming the boundary terms known and
we consider \eqref{eq:pairPDE} as a classical linear PDE, which can be solved via
the method of characteristics~\cite{Evans:2010}.
Such an approach yields integral representations for solutions to \eqref{eq:pairPDE}.
However, in addition to boundary terms, these solutions involve an indeterminate function 
which corresponds to constants of integration along the various characteristic curves.

To simplify the application of the method of characteristics, we perform a change of variables 
that transforms the pair-PDE \eqref{eq:pairPDE} into a PDE with constant first-order coefficients.
Specifically, we define the function $H$ by $H(x,y)=L(u,v)=L_{ij}(u,v)$ via the change of variable
\begin{eqnarray}\label{eq:varChange}
\begin{array}{ccc}
x&=& \displaystyle  \tau_i \ln{\left( 1+ \frac{u}{\tau_i}\right)} \vspace{5pt}  \\ 
y&=& \displaystyle \tau_j \ln{\left( 1+ \frac{v}{\tau_j}\right)}
\end{array}
\quad \mathrm{and} \quad
\begin{array}{ccc}
u&=& \displaystyle  \tau_i \left( e^{\frac{x}{\tau_i}}-1\right) \vspace{5pt}  \\ 
v&=& \displaystyle \tau_j \left( e^{\frac{y}{\tau_j}}-1 \right)
\end{array}
  \, .
\end{eqnarray}
Performing the above change of variables in the pair-PDE \eqref{eq:pairPDE} leads to a linear PDE for $H$ associated to a transport problem in the negative orthant $\mathbbm{R}^-\!\times\mathbbm{R}^-$
\begin{eqnarray}\label{eq:pairPDE2} 
\label{eq:pdeH}
\partial_x H + \partial_y H - f_{ij}(x,y) H = g_{ij}(x,y) \, ,
\end{eqnarray}
where we have introduced the auxiliary functions
\begin{eqnarray}
\label{eq:fij}
\lefteqn{ f_{ij}(x,y) = b_i \left( e^{\frac{x}{\tau_i}}-1\right) + b_j \left( e^{\frac{y}{\tau_j}}-1\right) } \nonumber\\
&& \hspace{110pt} + \sum_{k \neq i,j} \left( e^{ \tau_i \mu_{ik}  \left( e^{\frac{x}{\tau_i}}-1\right)+ \tau_j \mu_{jk}  \left( e^{\frac{y}{\tau_j}}-1\right)}-1\right) \beta_k \, ,
\end{eqnarray}
\begin{eqnarray}
\label{eq:gij}
\lefteqn{ 
g_{ij}(x,y) =  e^{ \tau_i r_i  \left( e^{\frac{x}{\tau_i}}-1\right)+ \tau_j \mu_{ji}  \left( e^{\frac{y}{\tau_j}}-1\right) } \partial_x H \vert_{x=0}  } \nonumber\\
&& \hspace{140pt}+ e^{ \tau_i \mu_{ij}  \left( e^{\frac{x}{\tau_i}}-1\right)+ \tau_j r_j \left( e^{\frac{y}{\tau_j}}-1\right)} \partial_y H \vert_{y=0} \, , 
\end{eqnarray}
with $\partial_x H (0,y)= \partial_u L (0,v) = \Exp{\lambda_i e^{v \lambda_j}}$.
Following on the analysis of the first-order RMF~\cite{Baccelli:2019aa}, we expect the solution to the pair-PDE
\eqref{eq:pairPDE} to admit an infinity of solutions $L$,
with possibly diverging behavior in $u=-\tau_i$ and $v=-\tau_j$.
After the change of variable, the corresponding loci for diverging behavior of $H$ are
$x \to -\infty$ and $y \to -\infty$, respectively. 
The following lemma specifies the integral representation of solutions to \eqref{eq:pdeH} 
that will form the basis for our analysis.

\begin{lemma}\label{lem:PDE}
The general solution to the linear first-order PDE \eqref{eq:pdeH} is given by
\begin{eqnarray}\label{eq:genPDE}
\lefteqn{ H(x,y) = e^{\int_0^x f_{ij}(u,y-x+u) \, du} } \nonumber\\
&& \hspace{50pt} \times \left( K(y-x) + \int_0^x g_{ij}(u,y-x+u) \, e^{-\int_0^u f_{ij}(v,y-x+v) \, dv} \, du \right) \, ,
\end{eqnarray}
where the function $K$ is determined as a boundary condition on the line $x=0$.
\end{lemma}

\begin{proof}
The characteristic curves $t \mapsto \big(\tilde{x}(t), \tilde{y}(t)\big)$ of the PDE \eqref{eq:pairPDE2} satisfy $\tilde{x}'(t)=\tilde{y}'(t)=1$.
Thus the characteristic curve passing through $(x, y)$ admits the parameterization:
\begin{eqnarray}
\tilde{x}(t)=t \quad \mathrm{and} \quad \tilde{y}(t)=t+y-x  \quad \mathrm{with} \quad \tilde{x}(x)=x \quad \mathrm{and} \quad \tilde{y}(x)=y\, .
\end{eqnarray}
Following the method of characteristics, we observe that the function defined by  $C(t)=H\big(\tilde{x}(t), \tilde{y}(t)\big)$ satisfies the one-dimensional linear ODE
\begin{eqnarray}
\partial_t C
=
f_{ij}\big(\tilde{x}(t), \tilde{y}(t)\big) C + g_{ij}\big(\tilde{x}(t), \tilde{y}(t)\big) \, ,
\end{eqnarray}
whose solution admits the following integral representation
\begin{eqnarray}
\hspace{25pt} C(t) = e^{\int_0^t f_{ij}(\tilde{x}(u), \tilde{y}(u)) \, du} 
\left( C(0) + \int_0^t g_{ij}(\tilde{x}(u), \tilde{y}(u))\, e^{-\int_0^u f_{ij}(\tilde{x}(v), \tilde{y}(v)) \, dv} \, du \right) \, ,
\end{eqnarray}
where the constant $C(0)$ only depends on $y-x$ via $C(0)=H\big(\tilde{x}(0), \tilde{y}(0)\big)=H(0,y-x)$.
Observing that $H(x,y)=H\big(\tilde{x}(x),\tilde{y}(x)\big)=C(x)$ and expressing that $\tilde{x}(u)=u$ and $\tilde{y}(u)=u+y-x$, the general solution to the PDE \eqref{eq:pdeH} is
\begin{eqnarray}
\lefteqn{H(x,y) = e^{\int_0^x f_{ij}(u,y-x+u) \, du} } \nonumber\\
&& \hspace{50pt} \times \left( K(y-x) + \int_0^x g_{ij}(u,y-x+u) \, e^{-\int_0^u f_{ij}(v,y-x+v) \, dv} \, du \right) \, ,
\end{eqnarray}
where $K$ is the boundary condition of $H$ on the line $x=0$, i.e., $K(y)=H(0,y)$.
In particular, we have $K(0)=H(0,0)=1$.
\end{proof}


\subsection{Integral representation on the boundary}\label{sec:bound}

The integral representation \eqref{eq:genPDE} for the functions $H$ solving \eqref{eq:pdeH}
involves boundary terms as unknowns as well as an undetermined function $K$.
Here, we further determine our PDE problem by deriving an alternative integral
representation for $H$ on the boundary $y=0$ that does not involve $K$.
The derivation of such an integral representation directly follows from considering the original PDE \eqref{eq:genPDE} as 
an ODE when restricted to the line $y=0$.  Specifically, we have:

\begin{lemma}\label{lem:ODE}
On the boundary $y=0$, the general solution to the linear first-order PDE \eqref{eq:pdeH} satisfies
\begin{eqnarray}
H(x,0) = e^{\int_0^x f_{ij}(u,0) \, du} \left( 1+ \int_0^x f_i(u) \, e^{-\int_0^u f_{ij}(v,0) \, dv} \, du \right) \, ,
\end{eqnarray}
where the functions $f_{ij}$ are defined in \eqref{eq:fij} and
the auxiliary function $f_i$ is defined by:
\begin{eqnarray}\label{eq:fi}
f_i(x) = g_{ij}(x,0) \left( 1 - e^{ \tau_i \mu_{ij}  \left(1- e^{\frac{x}{\tau_i}}\right)}\right) +  \beta_i e^{\tau_i (r_i-\mu_{ij})  \left( e^{\frac{x}{\tau_i}}-1\right) } \, ,
\end{eqnarray}
with
\begin{eqnarray}
\label{eq:newgij}
g_{ij}(x,0) = \beta_i e^{\tau_i r_i  \left( e^{\frac{x}{\tau_i}}-1\right) } +  e^{ \tau_i \mu_{ij}  \left( e^{\frac{x}{\tau_i}}-1\right) } \partial_y H (x,0) \, .
\end{eqnarray}
\end{lemma}

\begin{proof}
When specified on the boundary $y=0$, the PDE \eqref{eq:pairPDE2} reads
\begin{eqnarray}\label{eq:boundEq}
\partial_x H(x,0) + \partial_y H(x,0) - f_{ij}(x,0) H(x,0) = g_{ij}(x,0) \, ,
\end{eqnarray}
which we can interpret as an equation about $x \mapsto H(x,0)$, up to the term $\partial_y H(x,0)$.
The latter partial derivative term can be further expressed in terms of the coefficient functions of the PDE \eqref{eq:pairPDE2}.
Indeed, using the definition \eqref{eq:gij}, we can write the inhomogeneous term $g_{ij}(x,0)$ under the form \eqref{eq:newgij},
where we have utilized that $\partial_x H (0,0)= \partial_u L(0,0)=\Exp{\lambda_i}=\beta_i$.
In particular, relation \eqref{eq:newgij} defines the partial derivative $\partial_y H (x,0)$  as
\begin{eqnarray}
\partial_y H (x,0) = e^{ \tau_i \mu_{ij}  \left(1- e^{\frac{x}{\tau_i}}\right)} 
\left(g_{ij}(x,0) - \beta_i e^{ \tau_i r_i  \left( e^{\frac{x}{\tau_i}}-1\right) } \right) \, .
\end{eqnarray}
In turn, upon substitution in \eqref{eq:boundEq}, we obtain the following linear ODE for $x \mapsto H(x,0)$
\begin{eqnarray}\label{eq:2mom}
\partial_x H(x,0) = f_{ij}(x,0) H(x,0) + f_i(x),
\end{eqnarray}
where  the auxiliary function $f_i$ is defined as in  \cref{lem:ODE} in terms of the coefficient function $x \mapsto g_{ij}(x,0)$.
Bearing in mind that $H(0,0)=1$, the solution to the above equation admits the integral representation
\begin{eqnarray}
H(x,0) = e^{\int_0^x f_{ij}(u,0) \, du} \left( 1+ \int_0^x f_i(u) \, e^{-\int_0^u f_{ij}(v,0) \, dv} \, du \right) \, .
\end{eqnarray}
\end{proof}


\subsection{Simplification of the homogeneous solution}\label{sec:simple}

Assuming the boundary terms known, the method of characteristic yields an infinity of solutions $H$
to the PDE \eqref{eq:pdeH} considered on the negative orthant $\mathbbm{R}^- \! \times \mathbbm{R}^-$.
We are interested in solutions $(x,y) \mapsto H(x,y)$
representing MGF functions $(u,v) \mapsto L(u,v)$ via the smooth change of variables \eqref{eq:varChange}.
As MGF functions must be analytic on the negative orthant $\mathbbm{R}^- \! \times \mathbbm{R}^-$,
$L$ must be bounded at $u=-\tau_i$ and $v=-\tau_j$, which implies that 
$H$ must remain bounded when $x \to -\infty$ or $y \to-\infty$.
This requirement of boundedness for $H$ in $x \to -\infty$ imposes severe constraints on the undetermined function $K$ featured in the integral representation \eqref{eq:genPDE} obtained via the method of characteristic.
In fact, given boundary terms, we show that there is only one function $K$ ensuring the boundedness of $H$ in $x \to -\infty$ and that this function can be specified in terms of the integral representation of $H$ on the boundary $y=0$.
This leads to an integral representation for bounded solutions to \eqref{eq:pdeH} without indeterminacy, which we give in the following lemma:

\begin{lemma}\label{lem:compRep}
The solutions to the linear first-order PDE \eqref{eq:pdeH} which are bounded 
at $x=-\infty$, $y=-\infty $ admit the following integral representation
\begin{eqnarray}\label{eq:compRep}
\hspace{20pt} H(x,y) = e^{\int_0^x f_{ij}(u,y-x+u) \, du} \int_{-\infty}^x g_{ij}(u,y-x+u) \, e^{-\int_0^u f_{ij}(v,y-x+v) \, dv} \, du  \, .
\end{eqnarray}  
\end{lemma}

\begin{proof}
Specifying the general solution \eqref{eq:genPDE} obtained in  \cref{lem:PDE} on the line $y=0$ yields
\begin{eqnarray}\label{eq:auxBound}
\hspace{30pt} H(x,0) = e^{\int_0^x f_{ij}(u,u-x) \, du} \left( K(-x) + \int_0^x g_{ij}(u,u-x) \, e^{-\int_0^u f_{ij}(v,v-x) \, dv} \, du \right) \, ,
\end{eqnarray}
where $K$ is an unknown function representing the boundary condition of the linear first-order PDE \eqref{eq:pdeH}.
Utilizing the integral representation of $H(x,0)$ obtained in  \cref{lem:ODE}, the function $K$ can be expressed as
\begin{eqnarray}
\lefteqn{K(-x) = e^{\int_0^x \big( f_{ij}(u,0) - f_{ij}(u,u-x) \big) \, du} \left( 1+ \int_0^x f_i(u) \, e^{-\int_0^u f_{ij}(v,0) \, dv} \, du \right) } \nonumber\\
&& \hspace{200pt}- \int_0^x g_{ij}(u,u-x) \, e^{-\int_0^u f_{ij}(v,v-x) \, dv} \, du \, .
\end{eqnarray}
Substituting the above expression in \eqref{eq:auxBound} yields an expression that does not depend on the function $K$:
\begin{eqnarray}
\lefteqn{H(x,y) = e^{\int_0^x f_{ij}(u,y-x+u) \, du} \left( \int_{x-y}^x g_{ij}(u,y-x+u) \, e^{-\int_0^u f_{ij}(v,y-x+v) \, dv} \, du \right. } \nonumber\\
&& \hspace{50pt}+\left. e^{\int_0^{x-y} \big( f_{ij}(u,0) - f_{ij}(u,u+y-x) \big) \, du} \left( 1+\int_0^{x-y} f_i(u) \, e^{-\int_0^u f_{ij}(v,0) \, dv} \, du\right)\right) \, .
\end{eqnarray}
To further simplify our integral representation, let us introduce the new auxiliary function $G$ defined by 
\begin{eqnarray}
\lefteqn{G(x,z)  = \int_{z}^x g_{ij}(u,u-z) \, e^{-\int_0^u f_{ij}(v,v-z) \, dv} \, du  } \nonumber\\
&& \hspace{50pt}+ \; e^{\int_0^z \big( f_{ij}(u,0) - f_{ij}(u,u-z) \big) \, du} \left( 1+\int_0^z f_i(u) \, e^{-\int_0^u f_{ij}(v,0) \, dv} \, du\right) \, ,
\end{eqnarray}
which satisfies $H(x,y) = e^{\int_0^x f_{ij}(u,y-x+u) \, du} G(x,x-y)$.
The key observation is that for fixed $z$, we  have the asymptotic behavior
\begin{eqnarray}\label{eq:auxG}
\lim_{u \to \infty} f_{ij}(u,u-z) = - \left( b_i + b_j + \sum_{k \neq i,j} \left( e^{-\tau_i \mu_{ij} - \tau_j \mu_{ji}}-1\right) \beta_k\right) < 0 \, ,
\end{eqnarray}
which implies the divergence of the exponential factor intervening in the definition of $H$ in terms of $G$:
\begin{eqnarray}
\lim_{x \to -\infty} e^{\int_0^x f_{ij}(u,u-z) \, du} = \infty \, .
\end{eqnarray}
Thus, for the solution $H$ to remain bounded when $x \to - \infty$
and for all finite $z=x-y$, one must have that $\lim_{x \to -\infty} G(x,z) = 0$. 
Making this limit behavior explicit yields an integral equation about $g_{ij}$:
\begin{eqnarray}
& & \hspace{-1cm}
\int_{-\infty}^z g_{ij}(u,u-z) \, e^{-\int_0^u f_{ij}(v,v-z) \, dv} \, du\nonumber\\ & = &
e^{\int_0^z \big( f_{ij}(w,0) - f_{ij}(w,w-z) \big) \, dw} \left( 1+\int_0^z f_i(u) \, e^{-\int_0^u f_{ij}(v,0) \, dv} \, du\right) \, .
\label{eq:proto-fred}
\end{eqnarray}
Using the above integral equation in \eqref{eq:auxG} allows one to write the auxiliary function $G$ as
\begin{eqnarray}
G(x,z)  = \int_{-\infty}^x g_{ij}(u,u-z) \, e^{-\int_0^u f_{ij}(v,v-z) \, dv} \, du  \, ,
\end{eqnarray}
which implies the integral representation \eqref{eq:compRep} provided in  \cref{lem:compRep}.
\end{proof}


\subsection{System of integral equations for the boundary functions}\label{sec:intSys}

The integral representation \eqref{eq:compRep} of bounded solutions to the PDE \eqref{eq:pdeH}
features the auxiliary function $g_{ij}$ defined in \eqref{eq:gij}. 
The function $g_{ij}$ comprises the boundary terms of the PDE \eqref{eq:pairPDE2}, which involve
 the partial derivatives $\partial_x H \vert_{x=0}$ and $\partial_y H \vert_{y=0}$.
These partial derivatives admit a clear probabilistic interpretation via the Papangelou theorem of Palm calculus (see \Cref{sec:Palm}).
For instance, we have
\begin{eqnarray}\label{eq:propbInt}
\partial_y H (x,0)=\partial_vL(u,0) = \Exp{\lambda_j e^{u \lambda_i}} = \beta_j \Expj{ e^{u \lambda_i}} \, ,
\end{eqnarray}
showing that, up to the rescaling by the spiking rate $\beta_j$ and the change of variable given in
\eqref{eq:varChange}, the function $\partial_y H \vert_{y=0}$ is the MGF of $\lambda_i$ with respect to the Palm distribution associated to $N_j$.
A similar interpretation holds for $\partial_x H (0,y)$.
For conciseness, we denote the partial derivatives $\partial_y H \vert_{y=0}$ and
$\partial_x H \vert_{x=0}$ as the functions $h_i$ and $h_j$, respectively.
The functions $h_i$ and $h_j$ constitute the remaining unknowns of our problem,
as all other coefficient functions in \eqref{eq:compRep} have been elucidated.
In this section, we show that $h_i$ and $h_j$ satisfy a system of integral equations later specified in  \cref{lem:intSys}.

The first step of the derivation is to write the equation satisfied by $g_{ij}$ obtained in \eqref{eq:proto-fred}  as:
\begin{eqnarray}\label{eq:proto-fred2}
& & \hspace{-1cm} 1+\int_0^z f_i(u) \, e^{-\int_0^u f_{ij}(v,0) \, dv} \, du \nonumber\\
& &  =
e^{\int_0^z \big(f_{ij}(u,u-z) - f_{ij}(u,0)  \big) \, du} \int_{-\infty}^z g_{ij}(u,u-z) \, e^{-\int_0^u f_{ij}(v,v-z) \, dv} \, du \, . 
\label{eq:protof2}
\end{eqnarray}
Note that the function $f_i$ appearing in the above equation is actually defined in terms of $g_{ij}$ in \eqref{eq:fi}.
To simplify notation and exploit the symmetry of the problem, we further rewrite \eqref{eq:proto-fred2} 
using the following auxiliary functions
\begin{eqnarray}
\label{eq:mykl}
k_i(u) = e^{\tau_i r_i\left(e^{\frac{u}{\tau_i}}-1 \right)} \quad & \mathrm{and} &\quad k_j(u)=e^{\tau_j r_j\left(e^{\frac{u}{\tau_j}}-1 \right)} \, , \\
l_{ij}(u) = e^{\tau_i \mu_{ij} \left(e^{\frac{u}{\tau_i}}-1 \right)} \quad &  \mathrm{and} & \quad l_{ji}(u) = e^{\tau_j \mu_{ji} \left(e^{\frac{u}{\tau_j}}-1 \right)}  \, .
\end{eqnarray}
With these notations, we have
\begin{eqnarray}
g_{ij}(u,v) = h_j(v) l_{ji}(v)k_i(u) + h_i(u) l_{ij}(u)k_j(v),
\end{eqnarray}
so that \eqref{eq:protof2} can be written under a form involving two integral terms bearing on $h_i$ and $h_j$ respectively:
\begin{eqnarray}
 \lefteqn{1+\int_0^z f_i(u) \, e^{-\int_0^u f_{ij}(v,0) \, dv} \, du =} \nonumber\\
&& e^{-\int_0^z f_{ij}(u,0) \, du}
\left( \int_{-\infty}^z K_{ij}(z,u) h_i(u)\, du +  \int_{-\infty}^0 M_{ij}(z,u) h_j(u)\, du \right) \, .
\label{eq:protof3}
\end{eqnarray}
The integration kernels appearing in the right-hand side of the above equation are defined as
\begin{eqnarray}\label{eq:defK}
\begin{array}{ccc}
K_{ij}(z,u) &=& l_{ij}(u) k_j(u-z) e^{\int_u^z f_{ij}(v,v-z)\, dv}    \, , \\
M_{ij}(z,u) &=&  l_{ji}(u) k_i(u+z)   e^{\int_{u}^0 f_{ij}(v+z,v)\, dv} \, .
\end{array}
\end{eqnarray}
In turn, differentiating \eqref{eq:protof3} with respect to $z$ yields an equation without integral terms involving $f_i$
\begin{eqnarray} \label{eq:protof4a}
l_{ij}(z)h_i(z) &=& f_i(z) - \int_{-\infty}^z Q_{ij}(z,u) \, h_i(u) \, du - \int_{-\infty}^0   R_{ij}(z,u) \, h_j(u) \, du \, ,
\end{eqnarray}
at the cost of introducing the new integration kernels defined by
\begin{eqnarray}
\label{eq:myqr}
\begin{array}{ccccc}
Q_{ij}(z,u) 
&=& \partial_z K_{ij}(z,u) - f_{ij}(z,0) K_{ij}(z,u) \, , \vspace{5pt}  \\
&=& e^{\int_0^z f_{ij}(u,0) \, du} \partial_z \left[ K_{ij}(z,u)e^{-\int_0^z f_{ij}(u,0) \, du} \right] \, , \vspace{10pt} \\
R_{ij}(z,u) 
&=& \partial_z M_{ij}(z,u) - f_{ij}(z,0) M_{ij}(z,u)   \, , \vspace{5pt} \\
&=& e^{\int_0^z f_{ij}(u,0) \, du} \partial_z \left[ M_{ij}(z,u)e^{-\int_0^z f_{ij}(u,0) \, du} \right] \, .
\end{array}
\end{eqnarray}
By symmetry, considering the function $h_j$ yields a similar equation that reads
\begin{eqnarray} \label{eq:protof4b}
l_{ji}(z) h_j(z) &=& f_j(z) - \int_{-\infty}^0  R_{ji}(z,u) \, h_i(u) \, du - \int_{-\infty}^z  \ Q_{ji}(z,u) \, h_j(u) \, du \, .
\end{eqnarray}
As stated earlier, the functions $f_i$ and $f_j$ are defined in terms of $g_{ij}$ and thus involve boundary terms, i.e., $h_i$ and $h_j$.
In fact, using  results from  \cref{lem:ODE}, we have
\begin{eqnarray}
\begin{array}{ccccc}
f_i(z) &=& g_{ij}(z,0)\big(\displaystyle 1-1/l_{ij}(z)\big)+ \displaystyle \beta_i k_i(z)/l_{ij}(z) \, ,\vspace{5pt} \\
 &=&  h_j(0) k_i(z)+ (l_{ij}(z)-1)h_i(z) \, ,\vspace{10pt} \\
f_j(z) &=& g_{ij}(0,z)\big(\displaystyle 1-1/l_{ji}(z)\big)+ \displaystyle \beta_j k_j(z)/l_{ji}(z) \, ,\vspace{5pt} \\
&=&  h_i(0) k_j(z)+ (l_{ji}(z)-1)h_j(z) \, .
\end{array}
\end{eqnarray}
In both chains of equalities, the first equality is obtained by substituting \eqref{eq:newgij} in \eqref{eq:fi}, 
whereas the second equality stems from \eqref{eq:newgij} together with the fact that $h_i(0)=\beta_j$ and $h_j(0)=\beta_i$.
Expressing $f_i$ and $f_j$ in terms of $h_i$ and $h_j$ in  \eqref{eq:protof4a} and \eqref{eq:protof4b} produces the desired systems of integral equations, which we specify in the following lemma:

\begin{lemma}\label{lem:intSys}
The functions $h_i=\partial_y H \vert_{y=0}$  and $h_j=\partial_x H \vert_{x=0}$ satisfy the system of integral equations $h_i(z) = \mathcal{L}_i(h_i,h_j)$ and $h_j(z) = \mathcal{L}_j(h_i,h_j)$ where
\begin{eqnarray}\label{eq:linFredholm}
\begin{array}{ccc}
\hspace{25pt}  \mathcal{L}_i(h_i,h_j)= h_j(0) k_i(z) - \int_{-\infty}^z Q_{ij}(z,u) \, h_i(u) \, du - \int_{-\infty}^0   R_{ij}(z,u) \, h_j(u) \, du \, ,\\
\hspace{25pt}  \mathcal{L}_j(h_i,h_j)= h_i(0) k_j(z)  - \int_{-\infty}^0  R_{ji}(z,u) \, h_i(u) \, du - \int_{-\infty}^z  Q_{ji}(z,u) \, h_j(u) \, du \, ,
\end{array}
\end{eqnarray}
with
$h_i(0)=\beta_j$ and $h_j(0)=\beta_i$, $k_i$ and $k_j$ are defined in \eqref{eq:mykl}, and $Q$ and $M$ are defined in \eqref{eq:myqr}.
\end{lemma}

\begin{remark}
One can check that we have $Q_{ij}(0,u)=-R_{ji}(0,u)$ and $R_{ij}(0,u)=-Q_{ji}(0,u)$, which is consistent with the above equations specialized for $z=0$. 
\end{remark}

The system of integral equations \eqref{eq:linFredholm}  consists of two coupled homogeneous Fredholm equations of the second kind~\cite{Polyanin:2008}.
The kernels intervening in both Fredholm equations exhibit Dirac singularities in zero.
For such singular kernels, Fredholm theory does not guarantee the existence of nontrivial solutions~\cite{Edmunds:2018}.
However, in all generality,  the system of  equations \eqref{eq:linFredholm}  admits solutions if the linear operator $\mathcal{L} = (\mathcal{L}_i,\mathcal{L}_j)$ has a unit eigenvalue, in which case there is an infinite number of solutions given by the corresponding eigenfunctions.
Further specifying the solution to the PDE  \eqref{eq:pairPDE2} requires an additional constraint which can be obtained by specifying \eqref{eq:proto-fred2} for $z=0$, yielding:
\begin{eqnarray}\label{eq:norm}
 \lefteqn{1
=
 \int_{-\infty}^0 g_{ij}(u,u) \, e^{-\int_0^u f_{ij}(v,v) \, dv} \, du  
 = } \nonumber\\
 && \hspace{60pt} \int_{-\infty}^0 K_{ij}(0,u) h_i(u)\, du +  \int_{-\infty}^0 M_{ij}(0,u) h_j(u)\, du  \, . 
\end{eqnarray}
Note that by symmetry, we consistently have $K_{ij}(0,u)=M_{ji}(0,u)$ and $K_{ji}(0,u)=M_{ij}(0,u)$. 
If the unit eigenvalue is simple, there is at most one unit eigenfunction of $\mathcal{L}$
satisfying the above normalization condition.
Moreover, the interpretation of $h_i$ and $h_j$ in terms of MGFs
\eqref{eq:propbInt} further imposes analyticity constraints on candidate eigenfunctions to be solution.
Specifically, $h_i$ and $h_j$ must be completely monotonic functions when specified in the $(u,v)$ variables.
In \Cref{sec:probint}, we confirm that there is indeed a unique solution to the system of equations
\eqref{eq:linFredholm} satisfying the normalization constraint \eqref{eq:norm} by giving the
probabilistic interpretation of the problem in terms of the embedded Markov chains of the dynamics~\cite{Asmussen:2003}.


\section{The pair-transfer function}
\label{sec:transfer}


In this section, we exploit the reduction of the pair-PDE \eqref{eq:pairPDE} to the system of integral Fredholm equations \eqref{eq:linFredholm} and \eqref{eq:norm} to characterize the stationary state of a neuronal pair.
In particular, we study the stationary pair-transfer function, which specifies the output rates and output correlations as functions of the  input rates $\beta_k$, $k \neq i,j$ and of the connectivity weights $\mu_{ik}$, $k\neq i$ and $\mu_{jk}$, $k \neq j$.
In  \Cref{sec:statsec}, we determine how to derive the stationary second-order statistics 
from the stationary intensities via analytic arguments.
In  \Cref{sec:nodrive}, we analytically characterize the stationary state of an isolated neuronal pair.
In \Cref{sec:drive}, we perform the numerical analysis of the general case, i.e., subjected to independent Poissonian bombardment, via a fixed-point iterative scheme. 


\subsection{Stationary second-order statistics}\label{sec:statsec}

The core motivation for solving the pair-PDE \eqref{eq:pairPDE} is to account for pairwise correlation in LGL neural networks.
Second-order statistics are not directly accessible from our reduction of the problem to a set of integral Fredholm equations \eqref{eq:linFredholm}, where the mean intensities $\beta_i$ and $\beta_j$ plays the prominent role.
However, the second-order statistics of the state $\bm{\lambda} = \big(\lambda_i, \, \lambda_j\big)$ can be derived from the MGF function $L$, and thus from the transformed function $H$, as partial derivatives in $(0,0)$.
For instance, we have 
\begin{eqnarray}
\begin{array}{ccccc}
\Exp{\lambda_i^2} &=&\partial_u^2L(0,0)&=&\partial_x^2H(0,0) - \beta_i/\tau_i \, ,\vspace{5pt} \\ 
\Exp{\lambda_i \lambda_j} &=&\partial_u\partial_v L(0,0)&=&\partial_x\partial_y H(0.0) \, .
\end{array}
\end{eqnarray}
Thus, the second-order moments of the stationary intensities can be related to the mean intensities
$\beta_i$ and $\beta_j$ from the analysis of the PDE \eqref{eq:pairPDE2} performed in the prior sections.
Specifically, differentiating the linear ODE \eqref{eq:2mom} satisfied by $x \mapsto H(x,0)$ leads to an equation about $\partial_x^2H(0,0)$:
\begin{eqnarray}
\partial_x^2H(0,0) = \partial_x f_{ij}(0,0) + f_{ij}(0,0) \partial_xH(0,0) + \partial_x f_i(0) \, .
\end{eqnarray}
From the definition of $f_{ij}$ in \eqref{eq:fij} and the definition of $f_i$ in \eqref{eq:fi}, 
we observe that $f_{ij}(0,0)=0$ and we evaluate
\begin{eqnarray}
\partial_x f_{ij}(0,0) = \frac{b_i}{\tau_i} +\sum_{k \neq i,j} \mu_{ik} \beta_k \quad \mathrm{and} \quad \partial_x f_{i}(0) = \beta_j \mu_{ij} + \beta_i r_i \, .
\end{eqnarray}
This shows that the second-order moment of the stationary intensity $\lambda_i$ satisfies
\begin{eqnarray}\label{eq:secmom}
\Exp{\lambda_i^2} = \frac{b_i-\beta_i}{\tau_i} + r_i \beta_i +\sum_{j\neq i} \mu_{ij} \beta_j \, .
\end{eqnarray}
Remarkably, the functional dependence of $\Exp{\lambda_i^2}$ on the mean intensities $\beta_k$, $1 \leq k \leq n$, is the same as for first-order RMFs~\cite{Baccelli:2019aa}.

In principle, the mixed moment $\Exp{\lambda_i \lambda_j}$ can be derived from similar considerations about the transformed MGF $H$. 
Unfortunately, there seems to be no simple formula relating $\Exp{\lambda_i \lambda_j}$ to the mean intensities $\beta_i$ and $\beta_j$.
Rather, one has to observe that the mixed moment $\Exp{\lambda_i \lambda_j}$ is naturally expressed in
term of the boundary functions $h_i$ as we have $h'_i(0)=h'_j(0)=\partial_x\partial_y H(0,0)=\Exp{\lambda_i \lambda_j}$.
This leads to evaluate the mixed moments from the knowledge of $h_i$ and $h_j$ via the following integral equation
\begin{eqnarray}\label{eq:cross}
\hspace{15pt} h'_i(0) &=& \beta_j (r_i + r_j) - \int_{-\infty}^0 \partial_z Q_{ij}(0,u) \, h_i(u) \, du - \int_{-\infty}^0   \partial_z R_{ij}(0,u) \, h_j(u) \, du \, \, ,
\end{eqnarray}
obtained by differentiating the first equation of \eqref{eq:linFredholm} with respect to $z$ and by noticing that $Q_{ij}(0,0)=-r_j$. 
Observe that differentiating the second equation of \eqref{eq:linFredholm} yields the same equation, as one can check that $\partial_zQ_{ij}(0,u)=\partial_zR_{ji}(0,u)$ and $\partial_zR_{ij}(0,u)=\partial_zQ_{ji}(0,u)$.
Evaluating \eqref{eq:cross} is the approach we will take in \Cref{sec:drive} to numerically characterize the functional dependence of the pair correlations on input rates and synaptic weights.
However, before proceeding to this numerical analysis, we first give the full analytical solution of the pair-PDE  \eqref{eq:pairPDE} in  a simplifying limit obtained by neglecting relaxation, i.e., for $\tau_i, \tau_j \to \infty$, and under the additional assumption that the pair is isolated, i.e., with $\beta_k=0$, $k\neq i,j$.


\subsection{Exact solutions without external drive}\label{sec:nodrive}

\begin{figure}
  \centering
  \includegraphics[width=0.9\textwidth]{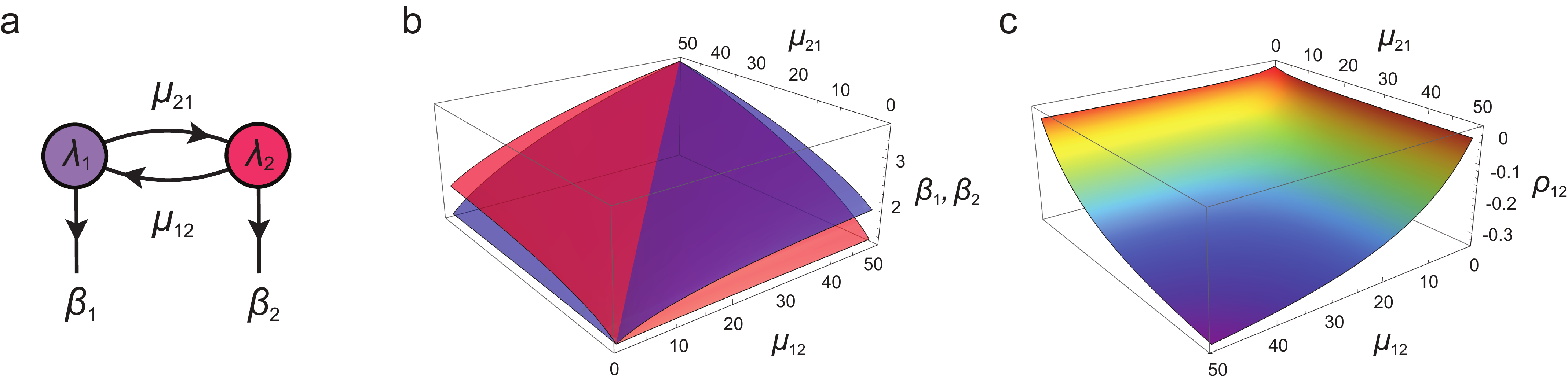}
  \caption{{\bf Stationary rates and correlations without external drive.}
  {\bf a.} An isolated pair of neurons $(1,2)$ interacting via synaptic weights $\mu_{21}$ and $\mu_{12}$,  with base rate $b_1=b_2=1$, and output stationary rate $(\beta_1,\beta_2)$.
  {\bf b.} Dependence of the stationary rates, i.e. of the mean stochastic intensities $(\beta_1,\beta_2)=(\Exp{\lambda_1},\Exp{\lambda_2})$, on the synaptic weights $\mu_{21}$ and $\mu_{12}$: $\beta_1>\beta_2$, whenever $\mu_{12}>\mu_{21}$.
  {\bf c.} Dependence of the stationary correlation between stochastic intensities, i.e. $\rho_{12} $, on the synaptic weights $\mu_{21}$ and $\mu_{12}$: the correlations are increasingly negative with increasing interaction strength because of the reset rule.
  }
  \label{fig:exact}
\end{figure}

Without external drive, i.e., for $\beta_k=0$, $k \neq i,j$, and without relaxation, i.e., for $\tau_i, \tau_j \to \infty$,
it is possible to solve exactly the system of integral equations \eqref{eq:linFredholm} 
and the normalization condition \eqref{eq:norm}.  Specifically, we have:

\begin{lemma}\label{lem:exactExpr}
For $\beta_k=0$, $k \neq i,j$, and for $\tau_i, \tau_j \to \infty$, the function $h_i$ solution of
the system equations \eqref{eq:linFredholm} and \eqref{eq:norm} is
\begin{eqnarray}
\hspace{25pt} h_i(z) = \frac{r_ir_j}{A_i r_i+A_j r_j -1} \frac{e^{\frac{r_j e^{\mu_{ij} z}}{\mu_{ij}}-(r_i+r_j)z} }{\mu_{ij}} \left(\frac{\mu_{ij}}{r_j}\right)^{\frac{r_i+r_j}{\mu_{ij}}} \gamma\left(\frac{r_i+r_j}{\mu_{ij}}, \frac{r_je^{\mu_{ij} z}}{\mu_{ij}}\right) \, ,
\end{eqnarray}
where $\gamma$ denotes the lower incomplete gamma function and where the constant $A_i$ and $A_j$ are given by
\begin{eqnarray}
A_i &=& \frac{e^{\frac{r_j}{\mu_{ij}}} }{\mu_{ij}} \left(\frac{\mu_{ij}}{r_j}\right)^{\frac{r_i+r_j}{\mu_{ij}}} \gamma\left(\frac{r_i+r_j}{\mu_{ij}}, \frac{r_j}{\mu_{ij}}\right) \, , \\
A_j &=& \frac{e^{\frac{r_i}{\mu_{ji}}} }{\mu_{ji}} \left(\frac{\mu_{ji}}{r_i}\right)^{\frac{r_i+r_j}{\mu_{ji}}} \gamma\left(\frac{r_i+r_j}{\mu_{ji}}, \frac{r_j}{\mu_{ji}}\right) \, .
\end{eqnarray}
\end{lemma}

\begin{proof}
Setting $\beta_k=0$, $k \neq i,j$ in expressions \eqref{eq:kerQ} 
yields the following form for the system of integral equations \eqref{eq:linFredholm}:
\begin{eqnarray}
\label{eq:Fredh1}
\hspace{20pt} h_i(z) &=& \beta_i e^{r_i z}  + r_j e^{-r_j z}\int_0^z   e^{(r_j + \mu_{ij}) u}\, h_i(u) \, du - r_i \int_0^1  e^{(r_i+\mu_{ji})u} \, h_j(u) \, du \, , \\
\hspace{20pt} \label{eq:Fredh2}
h_j(z) &=& \beta_j e^{r_j z}   - r_j \int_0^1  e^{(r_j+\mu_{ij})u} \, h_i(u) \, du + r_i e^{-r_i z}\int_0^z   e^{(r_i + \mu_{ji}) u}\, h_j(u) \, du \, .
\end{eqnarray}
The resolution of the above system is possible because the fixed bound integral terms are constants which we denote by
\begin{eqnarray}
C_i = \int_0^1  e^{(r_i+\mu_{ji})u} \, h_j(u) \, du \quad \mathrm{and} \quad
C_j = \int_0^1  e^{(r_j+\mu_{ij})u} \, h_i(u) \, du  \, .
\end{eqnarray}
These constants $C_i$ and $C_j$ are simply related via the normalization condition \eqref{eq:norm}:
\begin{eqnarray}\label{eq:exnorm}
C_i+C_j=1 \,  .
\end{eqnarray}
Moreover, specifying equation \eqref{eq:Fredh1} for $z=0$ and using that $h_i(0)= \beta_j$, we have
\begin{eqnarray}\label{eq:exspe}
\beta_j - \beta_i = r_j C_j - r_i C_i \, .
\end{eqnarray}
In turn, differentiating equation \eqref{eq:Fredh1} with respect to $z$ yields
\begin{eqnarray}
h'_i(z) 
&=& \beta_i r_i e^{r_i z}  + r_j e^{\mu_{ij}z} h_i(z) \\
&& \hspace{40pt}-r_j \left(r_j e^{-r_j z}\int_0^z   e^{(r_j + \mu_{ij}) u}\, h_i(u) \, du \right) -r_i^2e^{r_i z} C_i\, ,  \nonumber\\
&=& r_j \left( e^{\mu_{ij}z}-1\right) h_i(z)+(r_i+r_j) (\beta_i-r_i C_i)e^{r_i z} \, ,
\end{eqnarray}
where the second equality follows from injecting the expression of $h_i$ given by \eqref{eq:Fredh1}.
The above equation is a first-order linear differential equation whose unique bounded solution in
$z \to -\infty$ can be obtained by the method of the variation of the constant:
\begin{eqnarray}
h_i(z)= (r_i+r_j)(\beta_i-r_i C_i) \int_{-\infty}^z e^{r_i u + \int_u^z r_j \left( e^{\mu_{ij}v}-1\right) \, dv} \, du \, .
\end{eqnarray}
There remains to evaluate the constants $\beta_i$, $\beta_j$, and $C_i$, $C_j$.
To this end, let us introduce the constants
\begin{eqnarray}\label{eq:AiAj}
\hspace{30pt} A_i = \int_{-\infty}^0 e^{r_i u + \int_u^0 r_j \left( e^{\mu_{ij}v}-1\right) \, dv} \, du \, \quad \mathrm{and} \quad A_j = \int_{-\infty}^0 e^{r_j u + \int_u^0 r_i \left( e^{\mu_{ji}v}-1\right) \, dv} \, du \, .
\end{eqnarray}
Using that $h_i(0)= \beta_j$ and $h_j(0)= \beta_i$, we have
\begin{eqnarray}
\beta_j = (r_i+r_j)A_i (\beta_i - r_i C_i) \quad \mathrm{and} \quad \beta_i = (r_i+r_j)A_j (\beta_j - r_j C_j) \, .
\end{eqnarray}
Solving the above equations together with \eqref{eq:exnorm} and \eqref{eq:exspe} for $\beta_i$, $\beta_j$, $C_i$, $C_j$ yields:
\begin{eqnarray}
&\displaystyle \beta_i = \frac{A_j r_i r_j}{A_i r_i+A_j r_j -1}  \, , \quad \beta_j = \frac{A_i r_i r_j}{A_i r_i+A_j r_j -1} \, ,& \\
&\displaystyle \beta_i-r_iC_i= \beta_j-r_jC_j=\frac{r_i r_j}{(r_i+r_j)(A_i r_i+A_j r_j -1)} \, .&
\end{eqnarray}
The expression announced in the lemma follows from evaluating the following integral 
\begin{eqnarray}
\lefteqn{ \int_{-\infty}^z e^{r_i u + \int_u^z r_j \left( e^{\mu_{ij}v}-1\right) \, dv} \, du
 = } \nonumber\\
 && \hspace{60pt} \frac{e^{\frac{r_j e^{\mu_{ij} z}}{\mu_{ij}}-(r_i+r_j)z} }{\mu_{ij}} \left(\frac{\mu_{ij}}{r_j}\right)^{\frac{r_i+r_j}{\mu_{ij}}} \gamma\left(\frac{r_i+r_j}{\mu_{ij}}, \frac{r_je^{\mu_{ij} z}}{\mu_{ij}}\right)
\end{eqnarray}
in terms of the lower incomplete Gamma function $\gamma$ with
\begin{eqnarray}
A_i = \frac{e^{\frac{r_j}{\mu_{ij}}} }{\mu_{ij}} \left(\frac{\mu_{ij}}{r_j}\right)^{\frac{r_i+r_j}{\mu_{ij}}} \gamma\left(\frac{r_i+r_j}{\mu_{ij}}, \frac{r_j}{\mu_{ij}}\right) \, .
\end{eqnarray}
Observe that the derivative $h'_i(0) = h'_j(0)$ indicating the mixed moment $\Exp{\lambda_i \lambda_j}$ is given by: 
\begin{eqnarray}
h'_i(0) = h'_j(0)= \frac{r_ir_j}{A_i r_i+A_j r_j -1} \, .
\end{eqnarray}
\end{proof}

In  \cref{fig:exact}, we illustrate how the stationary intensities and stationary correlations of
an isolated pair of neurons $(1,2)$ depend on the synaptic weights $\mu_{12}$ and $\mu_{21}$. 
As expected, the stationary intensities increase with interaction strength and the neuron with highest intensity
is that for which the incoming synapse has larger weight than the outbound one's.
In general, isolated mutually exciting neuronal pairs exhibit negative stationary correlations.
Indeed, the definitions \eqref{eq:AiAj} implies that $0 \leq A_i r_i , A_j r_j\leq1$, so that the crosscorrelation coefficient $\rho$, which satisfies
\begin{eqnarray}
\rho=-\frac{r_ir_j}{A_i r_i+A_j r_j -1} ( A_i r_i -1)( A_j r_j -1) \leq 0 \, ,
\end{eqnarray}
is necessarily nonpositive.
These nonpositive correlations are due to the reset rule and the positive assumptions on interactions:
each time neuron $1$ spikes, its intensity resets to level $b_1=1$, while the other neuron's
intensity increases by an amount $\mu_{21}$. 
Accordingly, the larger the synaptic weights, the more negative the correlations.


\subsection{Integral equation solution with external drive}\label{sec:drive}

\begin{figure}
  \centering
  \includegraphics[width=\textwidth]{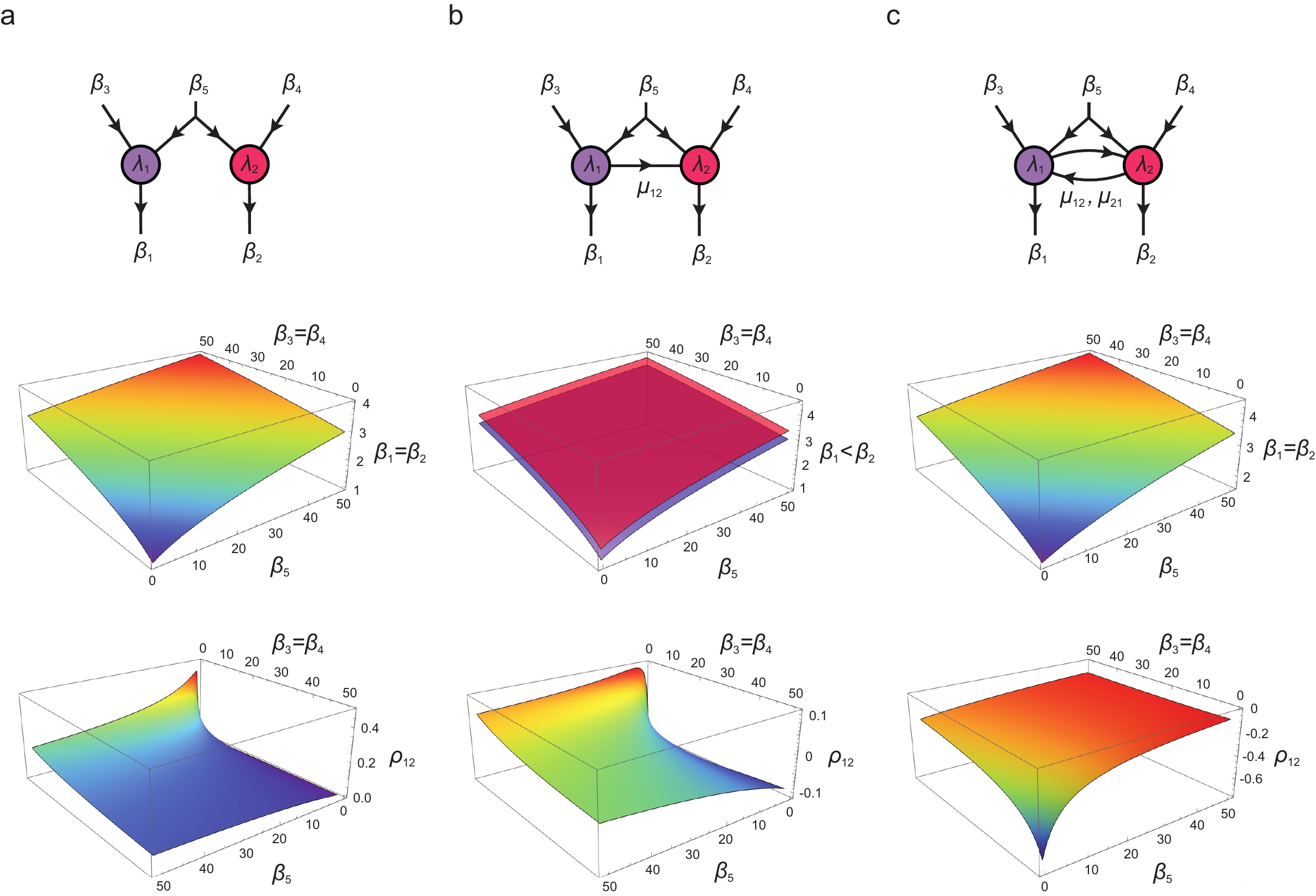}
  \caption{{\bf Stationary rates and correlations with external drive.}
 Stationary intensities $(\beta_1,\beta_2) = (\Exp{\lambda_1},\Exp{\lambda_2})$ and stationary correlations $\rho_{12}$ of a pair of neurons $(1,2)$ receiving private inputs from neurons $3$ and $4$ and shared input from neuron $5$, via synaptic unit weights.
  {\bf a.} Symmetric non-interacting pair: synchronous, shared input deliveries promotes positive correlations within the neuronal pair.
  {\bf b.} Non-symmetric interacting pair: correlations are positive when synchronous, shared input dominates ($\beta_5 \gg \beta_3,\beta_4$), whereas correlations are negative when independent private input dominates ($\beta_5 \ll \beta_3,\beta_4$).
  {\bf c.} Symmetric interacting pair: independent, private input deliveries promotes negative correlations within the neuronal pair.
  }
  \label{fig:num}
\end{figure}

In the presence of external drives, there are no known closed-form expressions for the functions $h_i$ and $h_j$ satisfying the system of Fredholm integral equations \eqref{eq:linFredholm} and the normalization condition \eqref{eq:norm}.
However, the latter equations naturally determine a fixed-point iterative algorithm,
which allows us to specify numerically $h_i$ and $h_j$, and thus the stationary intensities
$\beta_i$ and $\beta_j$, as well as the stationary second-order statistics.
We find this fixed-point iterative algorithm to be numerically stable, converging toward the unique solution of our problem.
However, the proof of such an algorithmic convergence is beyond the scope of this work. 

The algorithm proceeds as follows:
Let us denote by $h_{i,0}$ and $h_{j,0}$ our initial guesses for $h_i$ and $h_j$.
For instance, one can use the analytically known expressions for $h_i$ and $h_j$ in the absence of drive and interactions, i.e.:
\begin{eqnarray}
h_{i,0}(z) = r_j e^{b_i z} \quad \mathrm{and} \quad h_{j,0}(z) = r_i e^{b_j z} \, ,
\end{eqnarray}
with corresponding stationary intensities $\beta_{i,0}=h_{j,0}(0)=r_i$ and $\beta_{j,0}=h_{i,0}(0)=r_j$.
Then, for fixed external rates $\beta_k$, $k \neq i,j$, we define the sequence of functions $h_{i,n}$ and $h_{j,n}$
via the iterative scheme 
\begin{eqnarray}\label{eq:itscheme1}
h_{i,n+1}(z) &=& \nonumber\\
&& \hspace{-20pt} \left( \beta_{i,n} k_i(z) - \int_{-\infty}^z Q_{ij}(z,u) \, h_{i,n}(u) \, du - \int_{-\infty}^0   R_{ij}(z,u) \, h_{j,n}(u) \, du \right) \Big/\mathcal{N}_n\, , 
\end{eqnarray}
\begin{eqnarray}\label{eq:itscheme2}
h_{j,n+1}(z) &=&  \nonumber\\
&& \hspace{-20pt} \left( \beta_{j,n}k_j(z) - \int_{-\infty}^z  Q_{ji}(z,u) \, h_{j,n}(u) \, du - \int_{-\infty}^0  R_{ji}(z,u) \, h_{i,n}(u) \, du \right) \Big/\mathcal{N}_n \, ,
\end{eqnarray}
where the constant $\mathcal{N}_n$ follows from the normalization condition \eqref{eq:norm}:
\begin{eqnarray}\label{eq:itscheme3}
\mathcal{N}_n
 =
  \int_{-\infty}^0 K_{ij}(0,u) h_{i,n}(u)\, du +  \int_{-\infty}^0 M_{ij}(0,u) h_{j,n}(u)\, du  \, . 
\end{eqnarray}
The sequence of intensity estimates is consistently evaluated as 
\begin{eqnarray}\label{eq:itscheme4}
\beta_{j,n+1} = h_{i,n+1}(0) \quad \mathrm{and} \quad \beta_{i,n+1} = h_{j,n+1}(0)\, .
\end{eqnarray}
Converging sequences obtained via the above iterative scheme necessarily converge toward the unique solution to
the system of equations \eqref{eq:linFredholm} and \eqref{eq:norm}.
The stationary moments are obtained from the numerical stationary intensities via \eqref{eq:secmom},
whereas the stationary mixed moment is computed from the numerical functions $h_i$ and $h_j$ via \eqref{eq:cross}.

In  \cref{fig:num}, we illustrate how the numerical stationary intensities and the numerical stationary
correlations of a pair of neurons depend on the external drives.
Specifically, we consider a pair of neurons $(1,2)$ subjected to spike deliveries from three independent
Poissonian neurons via unit synaptic weights. 
Of these external neurons, two neurons deliver separately to a single neuron of the pair
with rates $\beta_3$ and $\beta_4$, representing private inputs, while the remaining neuron delivers
to both neurons with rate $\beta_5$, representing a shared input.
We analyze the contribution of private and shared inputs to the stationary intensities $(\beta_1,\beta_2)$ of a neuronal pair by varying the private rates $\beta_3=\beta_4$ and the shared rate $\beta_5$, for three distinct cases: no interaction (\cref{fig:num}a), unidirectional interaction (\cref{fig:num}b),
and symmetric interactions within the pair (\cref{fig:num}c).
For all conditions, and not surprisingly, the stationary intensities $(\beta_1,\beta_2)$ monotonically depend
on the external rates of spiking deliveries via positive synaptic rates.
By contrast, correlations depend more markedly on the interplay between external inputs and pairwise interactions.
In the absence of pairwise interactions, shared input promotes positive correlation,
whereas private inputs erased correlations (\cref{fig:num}a).
For unidirectional interactions with unit weight, increasing the rate of private inputs
can lead to negative correlations, which stems from post-spiking resets (\cref{fig:num}b).
For symmetric pairwise interactions with unit weight, negative correlations dominate
the dynamics irrespective of the input rates (\cref{fig:num}c).


\section{Pair-replica mean-field versions of LGL networks}
\label{sec:RI}

In this section, we discuss how to construct consistent second-order RMF models
for LGL neural networks using the solution of the pair PDE.
In \Cref{ssec:PhysRep}, we introduce the pair-RMF limits as physical systems
which naturally satisfy the Poisson Hypothesis by randomization of neural interactions.
In \Cref{ss:sce}, we state the self-consistency equations, which leverage the pair-PDEs
to characterize the stationary state of certain RMF limits, called pair-partition-RMF limits.
In \Cref{ss:sce}, we discuss the consistency issues that are present for another RMF limit
called the all-pair-RMF model. Both \Cref{ss:sce} and \Cref{ss:sce} include numerical
illustrations of our pair-RMF approach.


\subsection{Replica limits as physical models}\label{ssec:PhysRep}

\begin{figure}
  \centering
  \includegraphics[width=0.8\textwidth]{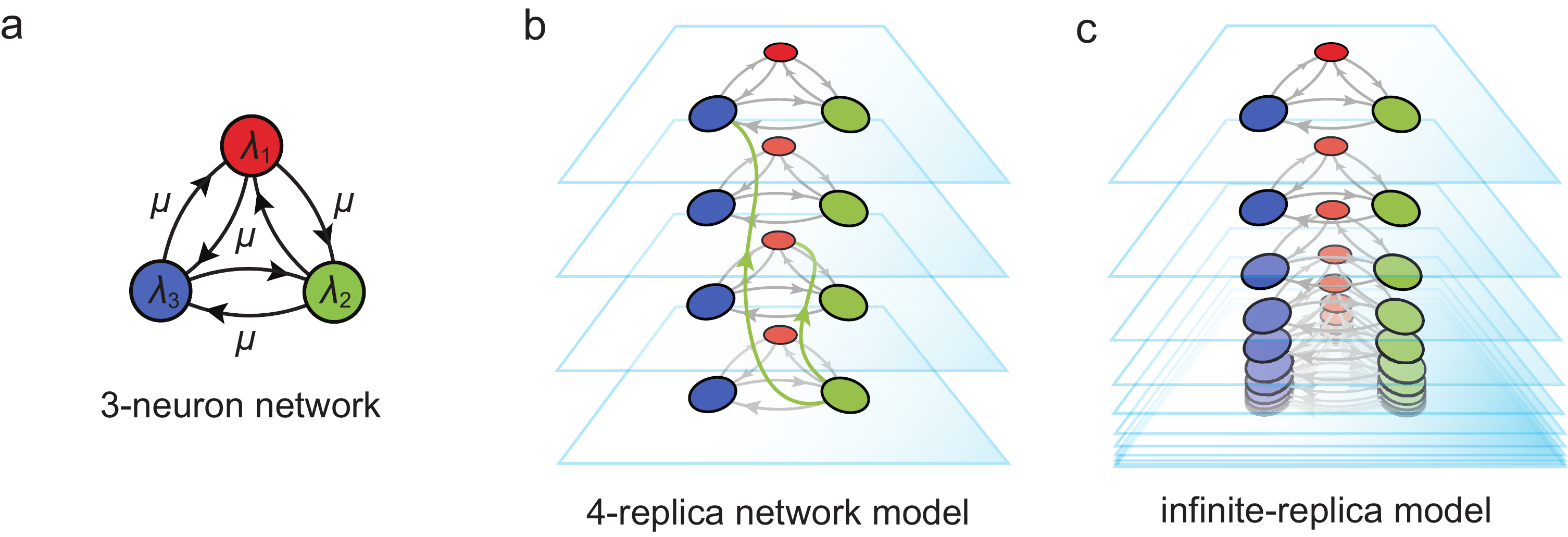}
  \caption{{\bf Physical RMF models.}
  {\bf a.} Original linear LGL networks of $K=3$ neurons.
  {\bf b.} First-order, finite-replica model with $M=4$ replicas. When a neuron spikes (green neuron), it interacts with downstream neurons sampled uniformly at random across replicas.
  {\bf c.} RMF models are obtained in the limit of an infinite number of replicas and represent infinite-size physical models supporting RMF dynamics.
  }
  \label{fig:repscheme}
\end{figure}

As already explained, the solution of the pair-PDE \eqref{eq:pairPDE} fully characterizes
the stationary state of a neuronal pair $(i,j)$ receiving independent Poissonian spike deliveries
from upstream neurons.  This corresponds to one of the simplest input-output networks for which
the state and outputs exhibit non-trivial correlations.
However, such a simple network is of limited interest in itself and
an important question is whether one can leverage analytic knowledge about
the pair system to study networks with a more complex structure.

In our introductory work \cite{Baccelli:2019aa}, we answer a similar question
under the most stringent Poisson Hypothesis which neglects all activity correlations
by introducing the so-called first-order RMF limits.
The merit of first-order RMF limits is to provide physical systems
that naturally implement the Poisson Hypothesis, and for which one can write self-consistency equations.
In turn, these self-consistency equations determine the stationary spiking rates throughout the whole network.
Just as for first-order RMF limits, we seek to state such self-consistency equations for pair-based RMF limits.
These pair-RMF limits introduced below can be seen as elaborations on the first-order RMF limits 
defined in \cite{Baccelli:2019aa}.

As depicted in \Cref{fig:repscheme}, first-order RMF limits comprise an infinite number
of replicas of the original systems. By original system, we mean the $K$-neuron network
with LGL dynamics described by~\eqref{eq:dynmod}.
First order $M$-replica models are obtained by considering that elementary constituents in
each replica are made of single neurons, whose autonomous dynamics is the same as in the original system.
However, the key difference with the original system is that upon spiking,
a neuron $i$ from replica $m$ interacts with neurons $(j,m'_j)$, $j \neq i$, via the
original weights $\mu_{ji}$ but in replicas $m'_j$,  $j \neq i$, chosen uniformly at random.
Intuitively, this randomization of interactions degrades statistical dependence
between neurons and across replicas. For an infinite number of replicas---in the RMF limit---,
the elementary constituents of replicas become asymptotically independent,
only interacting via self-consistently determined spiking rates $\beta_k$, $1 \leq k \leq K$.
Rigorously establishing this asymptotic independence for generic RMF limits is the object
of a forthcoming paper~\cite{Davydov:2020}.
Our goal here is to generalize the first-order RMF construction to include neuronal pairs,
thereby defining pair-RMF models.

Pair-RMF models are naturally obtained by allowing the elementary constituents of each
replica to be single neurons or pairs of neurons.
Upon spiking, neurons within a pair $(i,j)$ interact with one another according to the exact LGL dynamics,
but deliver spikes to neurons $(k,m_k)$, $k \neq i,j$, chosen uniformly at random across replicas.
Several pair-RMF models are possible due to the freedom to chose how one may define the various
replica constituents, i.e., whether a neuron appears as a single unit, as a pair member(s), or even as both.
In this work, we focus on two simple cases, the pair-partition and the all-pair cases,
and postpone the principled exploration of all pair-RMF limits to future work.  

Pair-partition-RMF models are those RMF limits for which elementary replica
constituents are either single neurons or pairs of neuron that form a partition of the set
of $K$ neurons of the original network. We denote the set pairs of the partition by $\mathcal{P}$ (with
elements $(i_1,j_1),\ldots,(i_p,j_p)$) 
and the set of singletons by $\mathcal{S}$ (with elements $k_1,\ldots,k_q$).
Then, denoting by $\lambda_{i,m}$ the stochastic intensity of neuron $i$ in replica $m$,
we have the following non-autonomous evolution for the $M$-replica dynamics of the network state
$\bm{\lambda}_{M} = \lbrace \lambda_{i,m} \rbrace_{1 \leq i \leq K, 1 \leq m \leq M}$:

\vspace{8pt}
\begin{itemize}
\item When a paired neuron $(i,m)$ spikes, the state variables $\bm{\lambda}_{M}$ change as follows: 
\begin{description}
\item[ Endogenous pair update:]{\ }\\ A spike is delivered to the matching neuron $j$ of the same replica
so that  $\lambda_{j,m} \leftarrow \lambda_{j,m}  + \mu_{ji}$, whereas the spiking
neuron $(i,m)$ resets to $r_i$.
\item[ Exogenous singleton updates:]{\ }\\ For all  $k\ne i,j$, a downstream replica $m_k$ is chosen uniformly
at random from $1,\ldots,m-1$, $m+1,\ldots,M$ so that $\lambda_{k,m_k}\leftarrow  \lambda_{k,m_k}  + \mu_{ki}$.
\item[ Exogenous pair updates:]{\ }\\ For all pairs $(i',j')\ne (i,j)$, a downstream replica $m_{i',j'}$
is chosen uniformly at random from the set $1,\ldots,m-1$, $m+1,\ldots,M$ so that
$\lambda_{i',m_{i',j'}} \leftarrow  \lambda_{i',m_{i',j'}}  + \mu_{i'i}$ and
$\lambda_{j',m_{i',j'}} \leftarrow  \lambda_{j',m_{i',j'}}  + \mu_{i'i}$.
\end{description}
\item When the paired neuron $(j,m)$ spikes, the symmetric update rule holds. 
\item When a singleton neuron $(k,m)$ spikes, the state variables $\bm{\lambda}_{M}$ 
change as follows:
\begin{description}
\item[ Endogenous singleton update:]{\ }\\ The spiking neuron $(k,m)$ resets to $r_k$.
\item[ Exogenous singleton updates:]{\ }\\ For all  singletons $l\ne k$, a downstream replica $m_l$ is chosen
independently at random  from $1,\ldots,m-1$, $m+1,\ldots,M$ so that we have
$\lambda_{k,m_l} \leftarrow  \lambda_{k,m_l}  + \mu_{ki}$.
\item[ Exogenous pair update:]{\ }\\ For all pairs $(i,j)$, a spike is delivered to a downstream
pair $(i,j)$ chosen independently at random
in replica $m_{i,j}$ from $1,\ldots,m-1$, $m+1,\ldots,M$ so that we have
$\lambda_{i,m_{i,j}} \leftarrow  \lambda_{i,m_{i,j}}  + \mu_{ik}$ and
$\lambda_{j,m_{i,j}} \leftarrow  \lambda_{j,m_{i,j}}  + \mu_{jk}$.
\end{description}
\end{itemize}
\vspace{8pt}

For exchangeable initial conditions, the various replicas are exchangeable. 
It should be intuitively clear (and we conjecture it on the basis of \cite{Davydov:2020}) that as $M$ tends to
infinity, the state variables of a replica have a distribution that tends to
a limit for weak convergence and that, in this limit, replicas become independent
with, in each replica, both the pair and all singletons subjected to independent Poissonian bombardment.
Thus, the one-pair-RMF model provides us with an infinite yet physical systems whose
dynamics satisfies the Poisson Hypothesis \cite{RybShlosI}.


\subsection{Consistency equations for pair-partition-RMF limits}\label{ss:sce}


Pair-par\-tition-RMF limits are physical models for which stationary input rates can be directly evaluated.
Such evaluation is made possible by the analytical treatment 
of the single neuron ODE in \cite{Baccelli:2019aa} and of the pair-PDE in \Cref{sec:APPDE}.
Here, we leverage these analytical treatments to specify the self-consistency rate
equations governing the stationary state of a network in the partition-RMF limit.

By construction, in a pair-partition-RMF model, each replica is partitioned in constituents
which are either single neurons and neuronal pairs. Each replica is partitioned in the same way.
Moreover, under the Poisson Hypothesis, each neuronal pair and each neuronal singleton is subjected to independent Poissonian bombardment.
As a result, the joint MGF of a replica state variable admits a product form
\begin{equation}
\label{eq:bigprod}
L (u_1,\ldots, u_K) =  \prod_{(i,j) \in \mathcal{P}} L_{ij}(u_i,u_j) \prod_{k \in \mathcal{S}} L_k(u_k) \,, 
\end{equation}
where $\mathcal{P}$ and $\mathcal S$ denotes the set of pairs and the set of singletons, respectively.
The elementary MGFs $L_{ij}$ and $L_k$ are determined as follows.

{\bf Single neuron:} 
If a neuron $k$ belongs to $\mathcal{S}$,
given external stationary rates $\beta_l$, $l \neq k$,
its stationary MGF $L_k$ satisfies the ODE
\begin{eqnarray}\label{eq:firstODE}
-\left( 1+ \frac{u}{\tau_k}\right)  L_k'(u) +  \left( \frac{ub_k}{\tau_k}  +\sum_{l \neq k}\left( e^{u \mu_{kl}}-1\right) \beta_l \right) L_k(u) +  \beta_k e^{u r_k} = 0 \, ,
\end{eqnarray}
and is known in closed form \cite{Baccelli:2019aa}.
Computing, $L_k$ from the above ODE assumes the knowledge of input rates $\beta_{[k]}= \{ \beta_l, l\ne k \}$.
Thus, $\beta_k$ is formally defined via a map that we denote
\begin{equation}
\label{eq:map1}
\beta_k= L'_k(0)\stackrel{\mathrm{def}}{=}F_{k}(\beta_{[k]}) \, .
\end{equation}

{\bf Pair of neurons:} If $(i,j)$ is a pair in $\mathcal{P}$, its stationary MGF $L_{ij}$
satisfies the pair-PDE \eqref{eq:pairPDE}.
This resolution involves computing the rates $\beta_i$ and $\beta_j$ via 
\eqref{eq:itscheme1}, \eqref{eq:itscheme2}, \eqref{eq:itscheme3}, \eqref{eq:itscheme4},
assuming knowledge of the external rates $\beta_{[ij]}=\{\beta_k, k\ne i,j \}$.
Thus,  $(\beta_i,\beta_j)$ are formally defined via a map that we denote
\begin{equation}
\label{eq:map2}
(\beta_i,\beta_j)= \left(\frac{\partial L_{ij}}{\partial u} (0,0) ,\frac{\partial L_{ij}}{\partial v} (0,0) \right)\stackrel{\mathrm{def}}{=} F_{ij}(\beta_{[ij]})\, .
\end{equation}

Determining the self-consistent MGF $L$ amounts to finding a solution
$\lbrace \beta_1,\ldots,\beta_k \rbrace$ to the system of self-consistency equations:
\begin{eqnarray}
\beta_k &= & F_{k}(\beta_{[k]}), \quad \forall \ k \in \mathcal{S},\\
(\beta_i,\beta_j) & = &  F_{ij}(\beta_{[ij]}),\quad \forall \  (i,j) \in \mathcal{P}.
\end{eqnarray}
The existence of a global solution is guaranteed by noticing that the corresponding physical
RMF dynamics satisfies the above consistency system of ODEs and PDEs under
the Poisson Hypothesis \cite{Davydov:2020}.
In practice, we find that a naive fixed-point iterative scheme always converges toward
the same solution for a given network. 
However, the question of the uniqueness of the pair-partition-RMF limit for
excitatory LGL networks remains open.


\begin{figure}
  \centering
  \includegraphics[width=\textwidth]{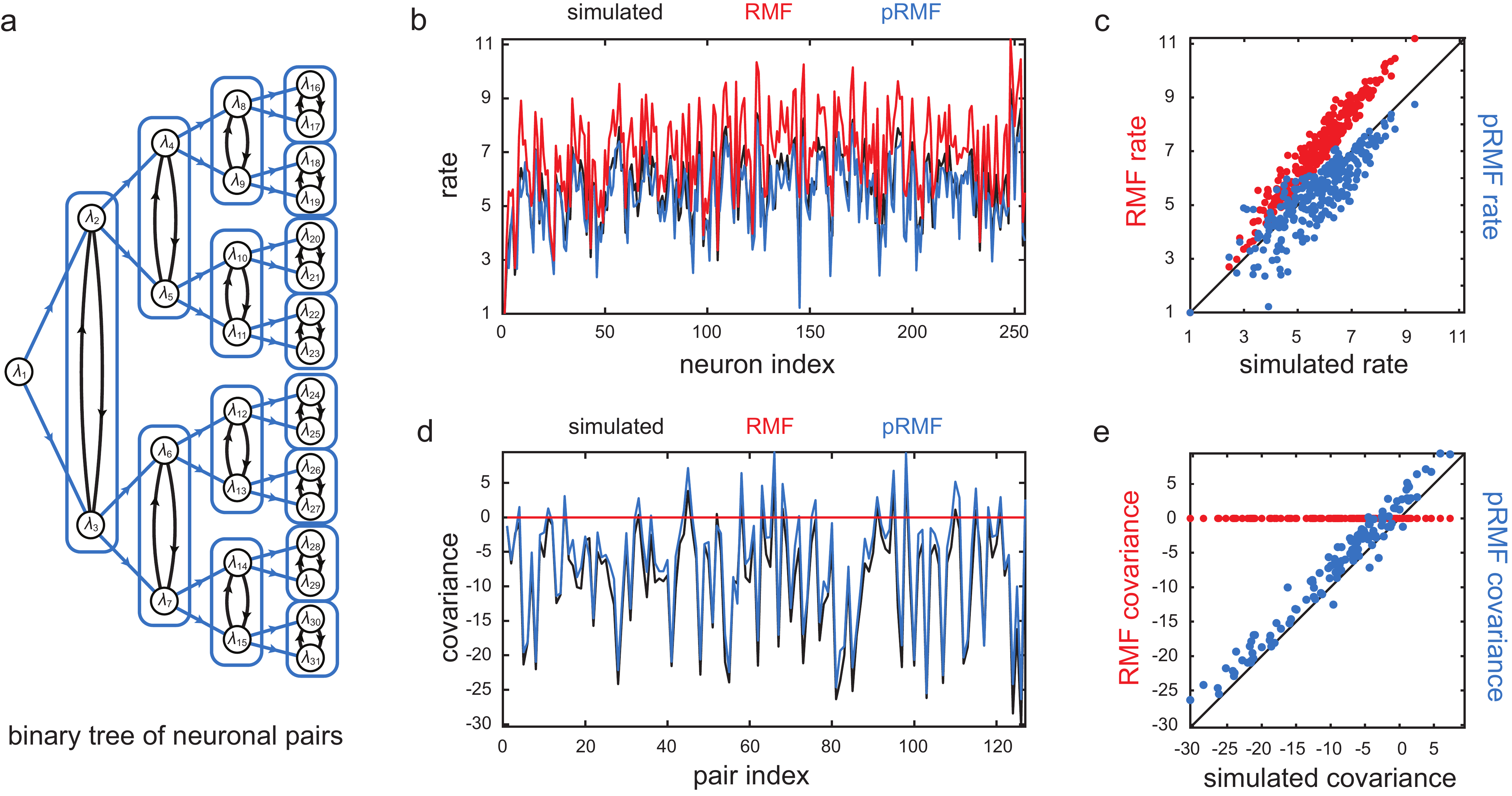}
  \caption{{\bf Pair-partition-RMF model.}
  {\bf a.}  Ensemble of pairs connected in a tree-organized feedforward circuit. Pair-partition-RMF models consider neuronal pairs (framed in blue) subjected to independent Poissonian spike deliveries (blue edges). Neuronal dynamics are devoid of relaxation with reset $r_i=1$ and connections weights $\mu_{ij}$ uniformly distributed in $(0,10)$.
  Simulations will be conducted for a tree of $7$ levels, i.e., $255$ neurons in $127$ pairs in addition to the root neuron.
  {\bf b.} Comparison between neuronal rates computed via exact event-driven simulations for the original model and rates computed via first-order RMF approach (RMF) and pair-partition-RMF approach (pRMF).
  {\bf c.} Scatter plot comparing the faithfulness of the first-order RMF approach and the 
pair-partition-RMF approach for stationary rates.
   {\bf d.} Comparison between pair-covariance estimates computed via exact event-driven simulations for the original model and pair-covariance estimates computed via first-order RMF approach (RMF) and pair-partition-RMF approach (pRMF).
  {\bf e.} Scatter plot comparing the faithfulness of the first-order RMF approach and the  
pai-partition-RMF approach for pair-covariance estimates.
  }
  \label{fig:partition}
\end{figure}

Figure \ref{fig:partition} illustrates numerically the pair-partition-RMF approach
for the case of a binary-tree feedforward structure, whereby children of a node interact as a pair.
\Cref{fig:partition}a depicts the overall structure of the original networks:
at each tree level, a neuron delivers spikes to a downstream pair of interacting neurons. 
All connection weights are randomly uniformly sampled, leading to an heterogeneous stationary regime.
The pair-partition-RMF model is obtained by considering all neuronal pairs (framed in blue)
as elementary constituents of the replicas, except for the root which is a singleton with no input.
To assess the faithfulness of the pair-RMF approach, we compare its spiking-rate and
pair-covariance estimates with first-order RMF estimates and with simulated estimates.
The latter estimates are computed via a discrete-event method using the Gillespie algorithm \cite{Gillespie:1977}.
\Cref{fig:partition}b and \Cref{fig:partition}c show that for a binary-tree feedforward structure,
the pair-partition-RMF approach marginally outperforms the first-order RMF ones.
\Cref{fig:partition}d and \Cref{fig:partition}e shows that the pair-partition-RMF
approach satisfactorily predicts the covariance among pairs, whereas by construction, 
the first-order RMF approach yields zero covariance between neurons.
These results provide an example of network structure for which the pair-partition-RMF
approach outperforms standard first-order RMF approach.
Remaining inaccuracies in predictions are due to approximating the activity
of upstream neurons as Poissonian: in the original network, neurons do not have
Poissonian activity as this only happens if their stochastic intensity is constant,
i.e., in the absence of interaction with other neurons.
We expect the pair-partition-RMF approach to perform well for network structures
involving strongly interacting neuronal pairs receiving Poissonian-like inputs.
Numerically, we find that this happens for sparse, feedforward connectivities,
such as the one presented in \Cref{fig:partition}.

\begin{figure}
  \centering
  \includegraphics[width=\textwidth]{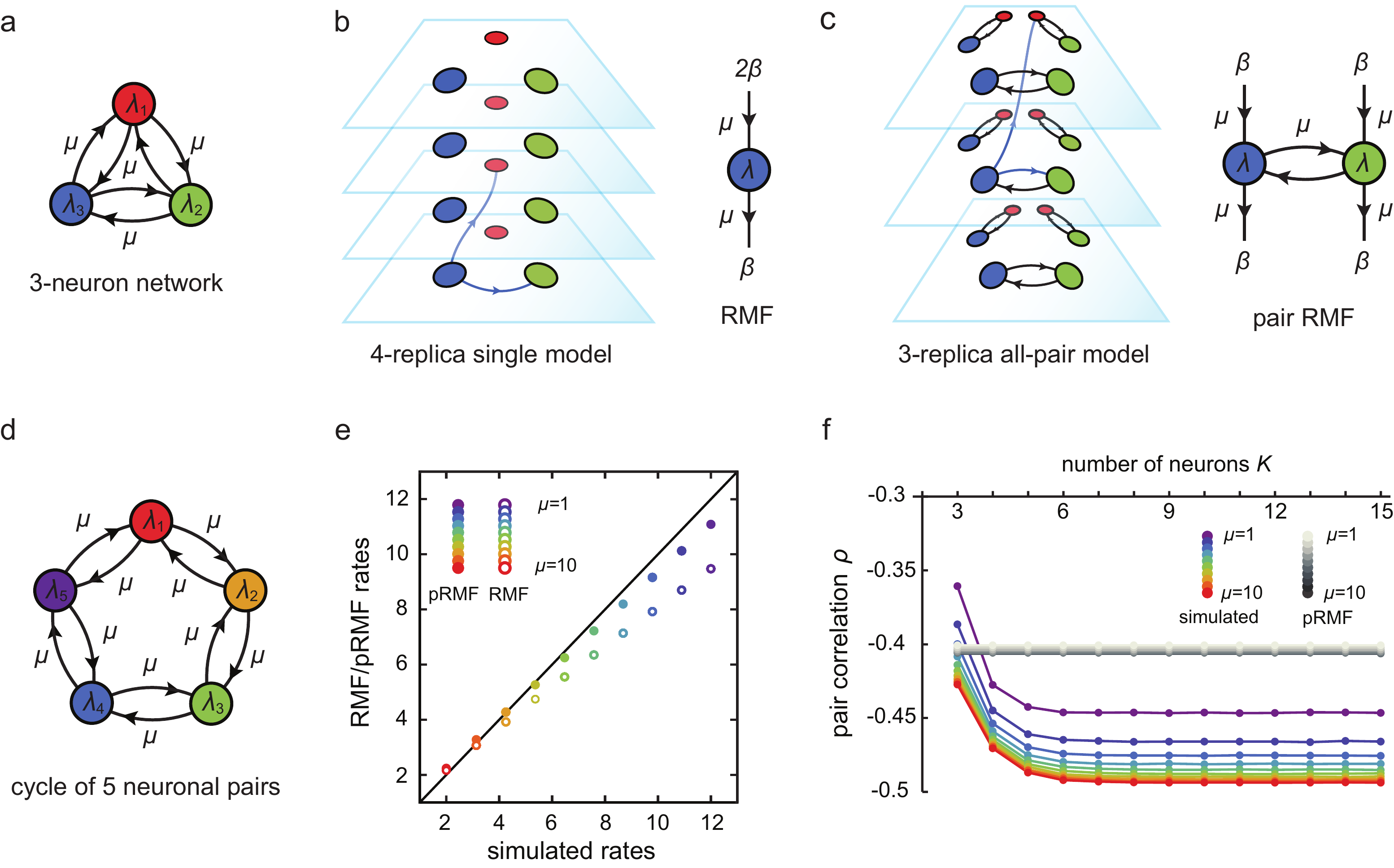}
  \caption{{\bf All-pair-RMF model.}
  {\bf a.} Original fully connected model of $K=3$ neurons.
  {\bf b.} The first-order RMF models consider that elementary replica constituents are isolated neurons.
 In homogeneous models, for which neurons are exchangeable, RMF limits can be resolved self-consistently (right schematic). 
  {\bf c.} The all-pair-RMF models consider that elementary replica constituents are made of all possible pairs.
  RMF limits can be resolved from appropriate self-consistency equations.
  {\bf d.} Homogeneous original network made of $K=5$ connected neuronal pairs without relaxation and reset $r=1$, and with variable connectivity $\mu$.
  {\bf e.} Scatter plot comparing the faithfulness of the first-order RMF approach and the  pair-partition-RMF approach for stationary rates, for $\mu=1, \ldots, 10$ and $K=3, \ldots 15$.
  {\bf f.} Comparison between pair-covariance estimates computed via exact event-driven simulations for the original model (color scale) and pair-covariance estimates computed via all-pair-RMF approach (gray scale). First-order RMF models yield zero correlations.
  }
  \label{fig:chain}
\end{figure}


\subsection{All-pair-RMF limits}

In all-pair-RMF models, each replica comprises all the possible pairs that can be formed by connected neurons.
This corresponds to considering the same neuron $i$ in different neighborhood contexts,
depending on which neuron $j$ engages in a pair with $i$.
We illustrate schematically this approach in \Cref{fig:chain}a-c by contrasting 
the first-order RMF and the all-pair-RMF approaches for an all-to-all network of $K=3$ exchangeable neurons.
The all-pair-RMF dynamics is best captured from its finite-replica version.
Specifically, let us denote by $\lambda_{i,m}^j$ the stochastic intensity of neuron $i$ in replica $m$ when paired with neuron $j$.
Then, we have the following non-autonomous evolution for the finite $M$-replica network state
$\bm{\lambda}_{M} = \lbrace \lambda^j_{i,m} \rbrace_{1 \leq i \neq j \leq K, 1 \leq m \leq M}$:

\vspace{8pt}
\begin{itemize}
\item  When a neuron $i$ from pair $(i,j)$ in replica $m$ spikes, the state variables $\bm{\lambda}_{M}$ change as follows: 
\begin{description}
\item[ Endogenous pair update:]{\ }\\ A spike is delivered to the matching neuron $j$ of the same replica
so that  $\lambda^i_{j,m} \leftarrow \lambda^i_{j,m}  + \mu_{ji}$, whereas the spiking
neuron resets: $\lambda^j_{i,m} \leftarrow r_i$.
\item[ Exogenous pair updates:]{\ }\\ For all neurons $j \ne i$, a target pair $(j,k_j)$ is chosen uniformly at random among the $(K-1)$ pair containing $j$ and a downstream replica $m_{j}$
is chosen uniformly at random from the set $1,\ldots,m-1$, $m+1,\ldots,M$, so that
$\lambda^{k_j}_{j,m_j} \leftarrow  \lambda^{k_j}_{j,m_j}  + \mu_{ji}$.
\end{description}
\item When the paired neuron $j$ from pair $(i,j)$ in replica $m$ spikes, the symmetric update rule holds. 
\end{itemize}
\vspace{8pt}

A feature of the all-pair-RMF approach is to yield multiple spiking rate estimates for a single neuron, which we denote by $\beta^j_i = \mathbbm{E}\big[\lambda^j_i\big]$ where $j$ denotes the paired neuron.
Moreover, within a given pair $(i,j)$, a neuron $i$ is bombarded across replica by neurons $k \ne i,j$ chosen uniformly at random across the $K-1$ possible pairs containing $k$, i.e., with aggregate rate:
\begin{eqnarray}\label{eq:aggregRate}
\beta_k = \frac{1}{K-1} \sum_{l \neq k} \beta^{(l)}_k \, .
\end{eqnarray}
Correspondingly, the self-consistency equations of the all-pair-RMF models read
\begin{equation}
\label{eq:allpaircons}
\left(\beta^{j}_{i},\beta^{i}_j \right)= F_{ij}(\beta_{[ij]}) \, , \quad \mathrm{for \; all} \quad  (i,j) \,,\ 1\le i<j\le K \, ,
\end{equation}
where $\beta_{[ij]} =\{\beta_k, k\ne i,j \}$ refers to the aggregate rates defined in \eqref{eq:aggregRate}.
Our replica interpretation guarantees the existence of a solution to the above system of $K(K-1)/2$ fixed-point equations.
By contrast with partition-pair-RMF solutions, this all-pair-RMF solution yields multiple estimates for the spiking rate of neuron $i$ via the rates $\beta^{j}_{i}$, $j \neq i$.
These rates specify the $K(K-1)/2$-dimensional MGF of a single all-pair-RMF replica $\Lambda$.
By construction, the MGF $L$ has the following product structure
\begin{eqnarray}
\Lambda \left(u^2_1, \ldots, u^K_1, \ldots, u^1_K, \ldots,u^{K-1}_K \right) = \prod_{i \neq j} L_{ij} \left( u^j_i,u^i_j \right) \, ,
\end{eqnarray}
where $L_{ij}$ solves the pair-PDE \eqref{eq:pairPDE} with external aggregate rates \eqref{eq:aggregRate}.
Then, one can easily define a consistent $K$-dimensional replica MGF by setting $u^j_i=u_i/(K-1)$ for all $j \neq i$ in the full MGF $\Lambda$ to obtain
\begin{eqnarray}
L(u_1, \ldots, u_K) = \prod_{i \neq j} L_{ij} \left( \frac{u_i}{K-1},\frac{u_j}{K-1} \right)  \, .
\end{eqnarray}
The function $L$ is simply the MGF of the stationary state of the average spiking rates $\bar{\lambda}_i = \sum_{j \neq i} \lambda^j_i/(K-1)$, whose means are precisely the aggregate rates $\beta_i$.
This shows that the rates $\beta_i$ corresponds to a consistent $K$-dimensional physical system with MGF $L$.
Although this model includes correlations, it remains a ``caricature'' of the original network.
Indeed, the rates $\beta^j_i$, and therefore the aggregate rates $\beta_i$, neglect the propagation of statistical dependencies beyond  single pairwise interaction.
This neglect effectively dampens the impact of correlations on rate estimates.


We illustrate this points by considering the network of
\Cref{fig:chain}d, which features a symmetric structure made of a circular
chain of $K$ neurons for which all neighboring neurons engage in pairs via homogeneous weights $\mu$.
Thus the only parameters specifying the structure are the numbers $K$ and $\mu$.
Observe first that RMF models do not depend explicitly on $K$ but on the connectivity 
number of the graph of interactions.
For instance, the symmetric all-pair-RMF models correspond to the following  boundary-value PDE problem:
\begin{eqnarray}\label{eq:allpairPDE}
\lefteqn{ \left( 1\!+ \! \frac{u}{\tau}\right)  \partial_u L+ \left( 1\!+ \! \frac{v}{\tau} \right) \partial_v L -\left( (u+v)\frac{b}{\tau}  + \left( e^{u \mu}+e^{v \mu}\!-\!2\right) \beta \right) \kappa L = } \nonumber\\
&&  \hspace{200pt} e^{u r+v \mu}  \partial_u L \vert_{u=0}+  e^{v r+u \mu}\partial_v L \vert_{v=0}  \, ,\\
&&  \hspace{110pt} \beta = \frac{\partial L}{\partial u} (0,0) = \frac{\partial L}{\partial v} (0,0) \, ,
\end{eqnarray}
where $\kappa$ is the connectivity number counting the number of external upstream neurons. 
For our chain-like model, we have $\kappa=1$.
Although the stationary spiking rates of the original model depend on $K$ in principle,
they quickly converge toward their infinite size limit $K \to \infty$ in practice.
Thus, the main determinant of the chain dynamics is the weight $\mu$.
In \Cref{fig:chain}e, we compare the stationary spiking rates for various $K$ and $\mu$
obtained for the first-order RMF and the all-pair-RMF models with rates obtained
via discrete-event simulation. Observe that the dependence on $K$ is barely noticeable
as rate estimates clustered around values solely determined by $\mu$.
Moreover, note that as expected, the all-pair-RMF model, which takes into account
all pairwise dependencies, outperforms the first-order RMF model.
Finally, in \Cref{fig:chain}f, we plot the pairwise correlation estimates for
the all-pair-RMF model (first-order-RMF approaches yield zero correlations).
We find that the simulated correlations depend more markedly on the length of the
chain $K$ than the simulated spiking rates. Moreover, we find that the correlation
values predicted by the all-pair-RMF model underestimate the simulated correlations.
This is consistent with the fact that dependencies do not propagate in all-pair-RMF models,
thereby leading to dynamics where the influences of correlations are dampened.


\appendix \label{sec:app}

\section{Primer on Palm calculus}\label{sec:Palm}
Palm calculus treats stationary point processes from the point of view of a typical point, i.e.,
a typical spike, rather than from the point of view of a typical time, i.e., in between spikes.
Here, we only introduce Palm calculus via the two formulae that will play a key role in deriving
the RMF {\it ansatz} \cite{BremaudBaccelli:2003}.
With no loss of generality, consider a stationary point process $N_i$ defined on some probability space
$(\Omega, \mathcal{F},\mathbb P)$, representing the spiking activity of a neuron.
If $\{\theta_t\}$ is a time shift on $(\Omega,{\mathcal F})$ which preserves $\mathbb P$, 
we say that the stationary point process $N$ is $\theta_t$-compatible in the sense that $N(B)\circ \theta_t=N(B+t)$ for all 
$B$ in $\mathcal{B}(\mathbb{R})$ and $t\in \mathbb R$.
With this notation, the Palm probability of $N$, which gives the point of view of a ``typical'' point on $N$,
is defined on $(\Omega, \mathcal{F})$ for all event $A$ in $\mathcal{F}$ and for all time $t>0$ by
\begin{eqnarray}\label{eq:Palm}
\hspace{15pt}
{\mathbb P}^0_N(A) = \frac{1}{\beta t} \Exp{\sum_{n \in \mathbb{Z}} \mathbbm{1}_A(\theta_{T_{n}})  \mathbbm{1}_{(0,t]}(T_{n})}
= \frac{1}{\beta t} \Exp{\int_{(0,t]} \left( 1_A \circ \theta_s \right) N(ds)} \, ,
\end{eqnarray}
where $\beta = \Exp{N((0,1])}$.
Informally, ${\mathbb P}^0_N(A)$ represents the conditional probability
that a train of spikes falls into $A$ knowing that a spike happens at $t=0$.
Moreover, suppose that $N$ admits a stochastic intensity $\lambda_i$,
representing the instantaneous spiking rate, and set $A= \{ \lambda(0) \in B\}$ for some $B$ in $\mathcal{B}(\mathbb{R}_+)$, then
\begin{eqnarray}
{\mathbb P}^0_N(A)  = {\mathbb P}^0_N \left[ \lambda (0_-) \in B  \right] = \mathbb P \left[ \lambda(0_-) \in B \, \vert \, N(\{0\})=1 \right] \, 
\end{eqnarray}
specifies the stationary law of the stochastic intensity $\lambda_i$ just before spiking.

The notions of Palm probability and stochastic intensity provide the basis for the theory of Palm calculus.
Let us consider another non-negative stochastic process $X$ defined on the same underlying probability 
space $(\Omega, \mathcal{F})$ as that of $N$.
If $X$ is also $\theta_t$-compatible in the sense that $X(s)\circ \theta_t=X(s+t)$ for all $t,s\in \mathbb R$,
then the first key formula Palm calculus directly follows from the definition \eqref{eq:Palm} and reads
\begin{eqnarray}\label{eq:Palm1}
\ExpPN{X(0_-)}
=
\frac{1}{\beta t} \Exp{\int_0^t X (s) N(ds)} \, ,
\end{eqnarray}
where $\ExpPN{ \cdot}$ denotes the expectation with respect to ${\mathbb P}^0_N$.
In the following, the process $X$ intervening in the above expression will typically
be a function of the stochastic intensity of a neuron.
The second key formula, which follows from the Papangelou theorem,
relates Palm probabilities to the underlying probability via the notion of stochastic intensity \cite{BremaudBaccelli:2003}.
Specifically, if $N$ admits a stochastic intensity $\lambda$ and $X$ has appropriate
predictability properties, then for all real valued functions $f$ we have:
\begin{eqnarray}\label{eq:Palm2}
\Exp{f(X(0)) \lambda_i(0)}
=
\beta \ExpPN{f\big(X(0_-)\big)}  \, .
\end{eqnarray}
The formulae \eqref{eq:Palm1} and \eqref{eq:Palm2} will be the only results required to
establish rate-conservation equations via Palm calculus.

We conclude by giving an application of Palm calculus which will be useful for the probabilistic interpretation of the integral equations \eqref{eq:linFredholm} and \eqref{eq:norm}.
Specifically, we consider the case of neuronal pair $(i,j)$ receiving independent Poissonian spike deliveries from upstream neurons.
Our goal is to evaluate the probability $\pi_i$ that neuron $i$ is the next one to spike under the stationary probability of the neuronal pair $(i,j)$.
Denoting by $\Expij{\cdot}$ the expectation with respect to the stationary process $N_i+N_j$, respectively, the probability $\pi_i$ is given by
\begin{eqnarray}
\hspace{20pt}
\pi_i
=
\Exp{\mathbbm{1}_{\lbrace T_{1,i}<T_{1,j}\rbrace}} 
=
\Expij{\mathbbm{1}_{\lbrace T_{1,i}<T_{1,j}\rbrace}} 
=
\frac{1}{\beta_i+\beta_j} \Exp{(\lambda_i+\lambda_j)\mathbbm{1}_{\lbrace T_{1,i}<T_{1,j}\rbrace}} \, ,
\end{eqnarray}
where the last formula follows from Papangelou theorem via \eqref{eq:Palm2}.
Using the key formula of Palm calculus \eqref{eq:Palm1}, the probability $\pi_i$ can be expressed as a stationary expectation $\Expi{\cdot}$ with respect to the process $N_i$:
\begin{eqnarray}
\pi_i
&=& \Expij{\mathbbm{1}_{\lbrace T_{1,i}<T_{1,j}\rbrace}} \\
&=& \frac{1}{\beta_i+\beta_j} \Exp{(\lambda_i+\lambda_j)\mathbbm{1}_{\lbrace T_{1,i}<T_{1,j}\rbrace}} \, ,\\
&=& \frac{\beta_i}{\beta_i+\beta_j} \Expi{\int_0^{T_i} \big(\lambda_i(t)+\lambda_j(t)\big) \mathbbm{1}_{\lbrace T_{1,i}<T_{1,j}\rbrace}\, dt } \, ,\\
&=& \frac{\beta_i}{\beta_i+\beta_j}  \Expi{\int_0^{T_{ij}} \big(\lambda_i(t)+\lambda_j(t)\big)\, dt } \, ,
\end{eqnarray}
where $T_{ij}$ is defined as $T_{ij}=\min(T_i,T_j)$, the survival time with hazard rate function $\lambda_i(t)+\lambda_j(t)$. 
Denoting by $\tilde{\lambda}_i$ the conditional hazard function of $T_{ij}$ under $\mathbb{P}_i^0$, one can finally evaluate
\begin{eqnarray}
\Expi{\int_0^{T_{ij}} \big(\lambda_i(t)+\lambda_j(t)\big)\, dt }
&=&
\int_0^\infty \left( \int_0^t \tilde{\lambda}_i(s) \, ds \right) \tilde{\lambda}_i(t) e^{- \int_0^t \tilde{\lambda}_i(s) \, ds} \, dt \\
&=&
\int_0^\infty \tilde{\lambda}_i(t) e^{- \int_0^t \tilde{\lambda}_i(s) \, ds}  \, dt = 1 \, ,
\end{eqnarray}
where the result follows from integration by parts.
This shows that $\pi_i= \beta_i/(\beta_i+\beta_j)$.

\section{Derivation of the pair PDE via the rate-conservation principle} 
\label{ssec:RCE}
Here, we establish the pair PDE \eqref{eq:pairPDE} bearing on the stationary MGF of an
interacting pair of neurons subjected to independent Poissonian spike deliveries.
This will require the use of the rate-conservation principle of Palm calculus for stationary
point processes~\cite{BremaudBaccelli:2003}.
Palm calculus treats stationary point processes from the point of view of a typical point, i.e.,
a typical spike, rather than from the point of view of a typical time, i.e., in between spikes.
We include a primer about Palm calculus in \Cref{sec:Palm} for the reader who is 
unfamiliar with this topic~\cite{Mathes:1964,Mecke:1967}.

The state variable of the neuronal pair are given by the stochastic intensities
$\bm{\lambda}_t=(\lambda_i(t),\lambda_j(t))$.
If the process  $\bm{\lambda}_t$ is $\mathcal{F}_t$-predictable for some filtration
$\{\mathcal{F}_t\}$ and if the dynamics of $\bm{\lambda}_t$ is stationary, then for
all real numbers $u$ and $v$, the process 
$\lbrace e^{u \lambda_i(t)+v \lambda_j(t)} \rbrace_{t \in \mathbb{R}}$
is also $\mathcal{F}_t$-predictable and stationary.
Moreover, this process satisfies the stochastic integral equation
\begin{eqnarray}
e^{u \lambda_i(t)+v \lambda_j(t)} - e^{u \lambda_i(0)+v \lambda_j(0)}
&=& \int_0^t \left( \frac{u}{\tau_i} \big(b_i \!-\! \lambda_i(s) \big)  + \frac{v}{\tau_j} \big(b_j \!-\! \lambda_j(s) \big)\right) e^{u \lambda_i(s)+v \lambda_j(s)} \, ds \nonumber \\
&& +\int_0^t \left(   e^{u r_i+v \mu_{ji}} - e^{u \lambda_i(s)}\right) e^{v \lambda_j(s)}  N_i(ds) \nonumber\\
&& +\int_0^t \left(   e^{v r_j+v \mu_{ij}} - e^{v \lambda_j(s)}\right) e^{u \lambda_i(s)}  N_j(ds) \nonumber \\
&&+ \sum_{k \neq i,j} \left( e^{u \mu_{ik}+v\mu_{jk}}-1\right) \int_0^t e^{u \lambda_i(s)+v \lambda_j(s)} N_k(ds) \, ,
\label{eq:relaxStoch}
\end{eqnarray}
where $N_i$ and $N_j$ are point processes with stochastic intensity $\lambda_i$ and $\lambda_j$, respectively,
and where $N_k$ are independent stationary Poisson processes with rate $\beta_k$.
In \eqref{eq:relaxStoch}, the first integral term is due to relaxation toward base rate $b_i$,
the next two terms are due to spiking of neurons $i$ and $j$ with regeneration at reset value
$r_i$ and $r_j$, respectively, and the last integral term is due to receiving spikes from
neurons $k \neq i,j$. 
Taking the expectation of \eqref{eq:relaxStoch} with respect to the stationary
measure of $\bm{\lambda}$ yields the following rate-conservation equation for
$\lbrace e^{u \lambda_i(t)+v \lambda_j(t)} \rbrace_{t \in \mathbb{R}}$:
\begin{eqnarray}\label{eq:rateCons1}
0&=& \frac{u}{\tau_i}\Exp{\int_0^t \big( b_i - \lambda_i(s)\big) e^{u \lambda_i(s)+v \lambda_j(s)} \, ds } +  \frac{v}{\tau_j}\Exp{\int_0^t \big( b_j - \lambda_j(s)\big) e^{u \lambda_j(s)} \, ds }   \nonumber\\
&& + \: \Exp{\int_0^t \left(   e^{u r_i+v \mu_{ji}} - e^{u \lambda_i(s)}\right) e^{v \lambda_j(s)}  N_i(ds)} \nonumber \\
&& + \: \Exp{\int_0^t \left(   e^{v r_j+v \mu_{ij}} - e^{v \lambda_j(s)}\right) e^{u \lambda_i(s)}  N_j(ds)} \nonumber \\
&& +  \sum_{k\neq i,j} \left( e^{u \mu_{ik}+v\mu_{jk}}-1\right) \Exp{\int_0^t e^{u \lambda_i(s)+v \lambda_j(s)} N_k(ds)  } \, ,
\end{eqnarray}
where we have used that by stationarity, 
$\Exp{e^{u \lambda_i(t)+v \lambda_j(t)}} = \Exp{e^{u \lambda_i(0)+v \lambda_j(0)}}$.
Again, by stationarity, the expectation of the relaxation integral terms can be expressed as 
\begin{eqnarray}
\Exp{\int_0^t \big( b_i - \lambda_i(s)\big) e^{u \lambda_i(s)+v \lambda_j(s)} \, ds } = t  \Exp{ (b_i-\lambda_i) e^{u \lambda_i+v \lambda_j}} \, .
\end{eqnarray}
In turn, introducing the Palm distribution ${\mathbb P}^0_i$
with respect to $N_i$ allows us to write the expectations of the remaining interaction and
reset integral terms as expectations with respect to the Palm distributions ${\mathbb P}^0_i$, $1 \leq i \leq K$.
Specifically, by applying formula \eqref{eq:Palm1}, we have
\begin{eqnarray}\label{eq:rateconv2a}
\lefteqn{ \Exp{ \int_0^t \left(   e^{u r_i+v \mu_{ji}} - e^{u \lambda_i(s)}\right) e^{v \lambda_j(s)}  N_i(ds)} =  } \nonumber\\
&& \hspace{100pt}\left( \beta_i t \right) \Expi{  \left( e^{u r_i+v \mu_{ji}} - e^{u \lambda_i(0^-)}\right) e^{v \lambda_j(0^-)} } \, ,
\end{eqnarray}
\begin{eqnarray}\label{eq:rateconv2b}
\Exp{\int_0^t e^{u \lambda_i(s)+v \lambda_j(s)} N_k(ds)  } &=& \left(\beta_k t \right) \Expk{ e^{u \lambda_i(0^-)+ v \lambda_j(0^-)}} \, , 
\end{eqnarray}
where $\beta_i = \Exp{\lambda_i}= \Exp{N_i((0,1])}$ is the mean intensity of $N_i$, and
$\Expi{\cdot}$ denotes  expectations with respect to ${\mathbb P}^0_i$.
With these observations, the rate-conservation equation \eqref{eq:rateCons1}
can be expressed under a local form, i.e., without integral terms, but at the cost of
taking expectation with respect to distinct probabilities:
\begin{eqnarray}\label{eq:rateConsPalm}
0 &=& \frac{u}{\tau_i}  \Exp{ (b_i-\lambda_i) e^{u \lambda_i+v \lambda_j}} +\frac{v}{\tau_j}  \Exp{ (b_j-\lambda_j) e^{u \lambda_i+v \lambda_j}} \nonumber\\
&& + \sum_{k\neq i,j}  \left( e^{u \mu_{ik}+v\mu_{jk}}-1\right) \beta_k \Expk{e^{u \lambda_i(0^-)+v \lambda_j(0^-)}}\nonumber\\
&& + \:\beta_i \Expi{  \left( e^{u r_i+v \mu_{ji}} - e^{u \lambda_i(0^-)}\right) e^{v \lambda_j(0^-)}} \nonumber\\
&& +  \: \beta_j \Expj{  \left( e^{v r_j+u \mu_{ij}} - e^{v \lambda_j(0^-)}\right) e^{u \lambda_i(0^-)}}  \, .
\end{eqnarray}
The above equation can then be expressed under a local form involving only
the stationary distribution of $\bm{\lambda}$ thanks to the hypotheses bearing on
external spiking process $N_k$, $k \neq i,j$, and to Papangelou's theorem \eqref{eq:Palm2}, 
Under our assumptions of stationary independent Poissonian deliveries,
Palm expectations with respect to the Poisson process $N_k$, $k \neq i,j$, are equivalent
to expectations with respect to the stationary distribution of $\bm{\lambda}$.
Intuitively, this follows from the fact that such external spiking deliveries sample
the dynamics of the neuronal pair at completely random, memoryless times.
Accordingly, we have 
\begin{eqnarray}
\hspace{30pt} \beta_k \Expk{  e^{u \lambda_i(0_-)+v \lambda_j(0^-)}} = \beta_k \Exp{ e^{u \lambda_i+v \lambda_j}}  \, .
\end{eqnarray}
In turn, Papangelou's theorem \eqref{eq:Palm2} allow us to write
\begin{eqnarray}
 \beta_i \Expi{  e^{u \lambda_i(0_-)+v \lambda_j(0^-)}} &=& \Exp{ \lambda_i e^{u \lambda_i+v \lambda_j}}  \, , \\
 \beta_j \Expj{  e^{u \lambda_i(0_-)+v \lambda_j(0^-)}} &=& \Exp{ \lambda_j e^{u \lambda_i+v \lambda_j}}  \, ,
\end{eqnarray}
for all $k$, $1 \leq k \leq K$. 
Using the above relations in \eqref{eq:rateConsPalm}, the final form of the rate-conservation equation of
$\lbrace e^{u \lambda_i(t)+v \lambda_j(t)} \rbrace_{t \in \mathbb{R}}$
becomes an equation about the stochastic
intensities $(\lambda_i,\lambda_j)$ involving stationary expectations only:
\begin{eqnarray}
0&=&\frac{u}{\tau_i}  \Exp{ (b_i-\lambda_i) e^{u \lambda_i+v \lambda_j}} +\frac{v}{\tau_j}  \Exp{ (b_j-\lambda_j) e^{u \lambda_i+v \lambda_j}}\nonumber\\
&& + \sum_{k\neq i,j}  \left( e^{u \mu_{ik}+v\mu_{jk}}-1\right) \beta_k  \Exp{ e^{u \lambda_i+v \lambda_j}} \nonumber\\
&&+ e^{u r_i+v \mu_{ji}}  \Exp{ \lambda_i e^{v \lambda_j}} - \Exp{ \lambda_i e^{u \lambda_i+v \lambda_j}} \nonumber\\
&&\nonumber\\
&&+e^{v r_j+u \mu_{ij}}  \Exp{ \lambda_j e^{v \lambda_j}} - \Exp{ \lambda_j e^{u \lambda_i+v \lambda_j}}  \, .
\end{eqnarray}
Interpreting the expectation terms in term of the MGF $L$ and its partial derivatives leads
to the pair-PDE \eqref{eq:pairPDE} appearing in \cref{def:pairPDE}.
As announced, this PDE characterizes the joint stationary distribution of the neuronal
pair $(i,j)$ as if bombarded by external neurons $k\neq i,j$ via independent stationary
Poissonian deliveries with rates $\beta_k$.


\section{Explicit forms without relaxation}\label{sec:norelax}
Here, we give the closed form expressions for the kernel functions involved in the system of integral equations \eqref{eq:linFredholm} and the normalization condition \eqref{eq:norm} in the absence of relaxation.
In the following, these expressions will be used for illustration of our numerical analysis.
However, the validity of our numerical analysis does not assume this simplifying limit and the expressions for kernels in the absence of relaxation are only specified for the sake of completeness.

Without relaxation, the stochastic intensities $\bm{\lambda}=(\lambda_i,\lambda_j)$ become pure jump processes
where the intensity components $\lambda_i$ and $\lambda_j$ are given by
\begin{eqnarray}
\lambda_i(t) = r_i + \sum_{k \neq i} C_{ik}(t) \quad \mathrm{and} \quad \lambda_j(t) = r_j + \sum_{k \neq j} C_{jk}(t),
\end{eqnarray}
where $C_{ij}$, $i \neq j$, are counting processes registering the number spike deliveries
from neuron $j$ to neuron $i$, since last time neuron $i$ spiked.
In other words, neglecting relaxation corresponds to considering neurons with perfect memory,
except for the post-spiking resets that erase prior spiking delivery counts.

In practice, the kernel functions without relaxation are obtained by taking the limit $\tau_i \to \infty$,
$\tau_j \to \infty$ in the relevant intervening quantities.
In this limit, the auxiliary function \eqref{eq:fij} mediating the external drive takes the simple form
\begin{eqnarray}
f_{ij}(x,y)= \sum_{k \neq i,j} \big( e^{\mu_{ik}x + \mu_{jk} y} -1 \big) \beta_k \, .
\end{eqnarray}
Upon integration of the above function, the kernel functions defined in \eqref{eq:defK} 
and appearing in the normalization condition \eqref{eq:norm} can be evaluated as
\begin{eqnarray}
K_{ij}(0,u) = M_{ji}(0,u)  = \exp \left(  (\mu_{ij}+r_j)u +\sum_{k \neq i,j} \beta_k \Big( u+c_{ijk}(0,u) \Big)\right) \, .
\end{eqnarray}
In turn, the kernel functions defined in \eqref{eq:myqr} and appearing in the system of integral equations
\eqref{eq:linFredholm}  can be evaluated as
\begin{eqnarray}\label{eq:kerQ}
\lefteqn{ Q_{ij}(z,u) = - \left( r_j+\sum_{k \neq i,j} \beta_k \mu_{jk}c_{ijk}(z,u)  \right) } \nonumber\\
&& \hspace{100pt} \times   \exp \left(  r_j(u-z) + \mu_{ij} u +\sum_{k \neq i,j} \beta_k \Big( u-z+c_{ijk}(z,u) \Big)\right) 
\end{eqnarray}
\begin{eqnarray}\label{eq:kerR}
R_{ij}(z,u) &=&    \left( r_i+\sum_{k \neq i,j} \beta_k \Big( 1-e^{z \mu_{ik}} + \mu_{ik}d_{ijk}(z,u) \Big) \right)  \nonumber\\
&&  \hspace{60pt} \times   \exp{\left(  r_i (u+z) + \mu_{ji} u +\sum_{k \neq i,j} \beta_k \Big( u+ d_{ijk}(z,u) \Big) \right)} 
\end{eqnarray}
with auxiliary functions
\begin{eqnarray}
c_{ijk}(z,u)&=& \frac{  e^{z \mu_{ik}}}{\mu_{ik}+\mu_{jk}} \left(1- e^{(u-z)(\mu_{ik}+\mu_{jk})}\right) \, , \\
\quad d_{ijk}(z,u)&=& \frac{  e^{z \mu_{ik}}}{\mu_{ik}+\mu_{jk}} \left(1- e^{u(\mu_{ik}+\mu_{jk})}\right) \, .
\end{eqnarray}
Taking into account relaxation yields integral kernels involving a special function, namely, the exponential integral function.
However, for the sake of simplicity, we will only consider the case without relaxation to illustrate our method numerically.




\section{Probabilistic interpretation}\label{sec:probint}
In this section, we show that the system of integral equations \eqref{eq:linFredholm},
as well as the normalization condition \eqref{eq:norm}, receive a probabilistic 
interpretation in terms of the stationarity of the embedded Markov chain associated
to the continuous time Markovian dynamics.
The existence and uniqueness of a solution to that system of equation is then a 
direct consequence of the ergodicity of the dynamics of the $(i,j)$-pair subjected
to independent Poissonian bombardment from other neurons.
In  \Cref{sec:emc}, we define two tightly-related embedded Markov chains for a 
pair of interacting neurons as the sequence of instantaneous pre-spiking and post-spiking
stochastic intensities of the neuronal pair.
In  \Cref{sec:condmgf}, we give an integral representation for the conditional
moment-generating functions of the embedded Markov chain of a neuronal pair.
In \Cref{sec:intec}, we utilize the obtained integral representation to show
that the invariant measure of the embedded Markov chain satisfies conservation equations
equivalent to the system \eqref{eq:linFredholm} .
In  \Cref{sec:normcond}, we finally show that the normalization condition
\eqref{eq:norm} follows from stationarity of the counting process registering the spiking of the pair.


\subsection{Embedded Markov chain}\label{sec:emc}

The joint stochastic intensities $\bm{\lambda} = \big(\lambda_i, \, \lambda_j\big)$
form the state of the continuous-time Markovian dynamics of the $(i,j)$-pair of neurons.
In the RMF limit, upstream neurons $k$, $k\neq i, j$, deliver spikes to the $(i,j)$-pair according to independent Poisson processes with mean intensities $\beta_k$, $k \neq i,j$.
As a consequence, the coupling between neuron $i$ and $j$ is entirely due to interactions within the pair and dependencies do not propagates via relay neurons.
In turn, the stationary law of these pairwise interactions can be characterized via an embedded Markov chain obtained by specializing the continuous-time Markovian dynamics at pairwise interaction events, i.e., whenever neuron $i$ or neuron $j$ spikes.

Due to the instantaneous nature of the interactions, there are two possible choices for the embedded Markov chains, depending on whether one considers stochastic intensities just before spiking times or just after spiking times.
The pre-spiking chain $\lbrace \bm{\lambda}^-_n\rbrace_{n \in \mathbb{Z}}$ and the post-spiking 
chain $\lbrace \bm{\lambda}^+_n\rbrace_{n \in \mathbb{Z}}$ are defined by $\bm{\lambda}^-_n=\bm{\lambda}(T_{ij,n}^-)$
and $\bm{\lambda}^+_n=\bm{\lambda}(T_{ij,n}^+)$ respectively, where $\lbrace T_{ij,n}\rbrace_{n \in \mathbb{Z}}$
denotes the ordered sequence of spiking events of the $(i,j)$-pair.
The stationary probabilities of both chains are related via the post-spiking reset rules.
This relation is naturally expressed in terms of the Palm distributions of the stochastic intensities $\bm{\lambda}$.
The Palm probability of a stationary point process can be interpreted as the distribution of
this point process conditioned to have a point present at the origin of the time axis (see Section \Cref{sec:Palm}).
We denote by $\mathbb{P}^0_{i}$ and $\mathbb{P}_j^0$ the Palm probabilities associated to the spike
counting processes $N_i$ and $N_j$, respectively, and by $\mathbb{P}^0_{ij}$ the Palm probability associated to $N_i+N_j$,
when either neuron of the $(i,j)$-pair spikes. 
Here these Palm probabilities are with respect to the steady state
of the $(i,j)$-RMF model.
Under $\mathbb{P}^0_i$, $\lambda_j^-=\lambda_j(0^-)$ has a density on $[b_j,\infty)$ that we denote by $p^0_i$.
Similarly, we denote by $p^0_j$ the density of $\lambda_i^-=\lambda_i(0^-)$ under $\mathbb{P}^0_j$.
Because post-spiking resets erase all information about the stochastic intensities of the spiking neurons,
the distributions $p^0_i$ and $p^0_j$ carry all the relevant information about the coupled dynamics of the neuronal pair.
To see this, let us introduce $\pi_i$ and $\pi_j$, the stationary probabilities that neuron $i$ or neuron $j$
spike given that the pair spikes.
These probabilities are simply defined as ratios of stationary rates (see Section \Cref{sec:Palm}):
\begin{eqnarray}\label{eq:defPi}
\pi_i = \frac{\beta_i}{\beta_i+ \beta_j} \quad \mathrm{and} \quad \pi_j = \frac{\beta_j}{\beta_i+ \beta_j} \, ,
\end{eqnarray}
The joint stationary probability of $\bm{\lambda}^+=\bm{\lambda}(0^+)$ under $\mathbb{P}^0_{ij}$ is then given by:
\begin{eqnarray}\label{eq:consDistr}
p^0_{ij}(d \lambda_i,d\lambda_j)
= \pi_i \delta_{r_i}(d\lambda_i) \, p^0_i(\lambda_j-\mu_{ji})\, d\lambda_j +
\pi_j p^0_j(\lambda_i-\mu_{ij}) \, d\lambda_i \, \delta_{r_j}(d\lambda_j) \, .
\label{eq:palmdepalm}
\end{eqnarray}
In the above relation, the Dirac delta terms follow from post-spiking reset to values $r_i$ and $r_j$,
whereas the weight-shifted densities follow from post-spiking interactions within the pair.

Our goal is to derive integral equations satisfied by $p^0_{ij}$, and thus by  $p^0_i$ and $p^0_j$,
from conservation laws about the embedded chains.
These conservation laws follow from the invariance of the Palm distribution $\mathbb{P}^0_{ij}$
with respect to time shifts from one spiking event of the pair to the next spiking event.
Specifying the equations attached to these conservation laws requires to express the conditional
probability of the next spiking time of neuron $i$ or $j$ given the state $\bm{\lambda}$ of the $(i,j)$-pair
and given the history of external spike deliveries to the $(i,j)$-pair.
Let $T_{i,1}$ and $T_{j,1}$ denote the first positive points of $N_i$ and $N_j$, i.e.,
the first spiking time of neuron $i$ and neuron $j$, respectively.
For fixed $t>0$, let also $\bm{N} = \lbrace N_k(s) \rbrace_{0 \leq s \leq t, k \neq i,j}$
denote the history of external spike deliveries to the $(i,j)$-pair up to time $t$.
By definition of the stochastic intensities, we have
\begin{eqnarray}\label{eq:condTrans}
\lefteqn{ \mathbb{P} \left[t-dt<T_{j,1}<t, T_{i,1}>t \, \vert \, \bm{\lambda}(0) , \, \bm{N}\right] = } \nonumber\\
&& \hspace{40pt} \lambda_j(t \, \vert \, \bm{\lambda}(0),\bm{N}) dt
\exp{\left(-\int_0^t \Big( \lambda_i(s\, \vert \, \bm{\lambda}(0) , \, \bm{N}) +\lambda_j(s \, \vert  \, \bm{\lambda}(0) , \, \bm{N}) \Big) \,ds \right)},
\end{eqnarray}
where $\mathbb P$ is the law of the Markov process $(\bm{\lambda}(t),\bm{N}[0,t])$ and where
the neuron-specific stochastic intensities satisfy
\begin{eqnarray}\label{eq:stochIntForm}
\lambda_i(t \, \vert \,  \bm{\lambda}(0),  \bm{N}) = \tilde{\lambda_i}(t) + \sum_{k \neq i,j} \lambda_{ik}(t \, \vert \, N_k) \, ,
\end{eqnarray}
with state-dependent part $\tilde{\lambda_i}$ and externally-driven part $\lambda_{ik}$:
\begin{eqnarray}
\tilde{\lambda_i}(t) =b_i+(\lambda_i(0)-b_i)e^{-t/\tau_i}  
\quad \mathrm{and} \quad \lambda_{ik}(t \, \vert \, N_k) = \mu_{ik} \sum_{l: 0< T_{k,l} \leq t} e^{\frac{T_{k,l}-t}{\tau_i}} .
\end{eqnarray}
Expressions \eqref{eq:condTrans} are valid whenever the initial state $\bm{\lambda}(0)$
and the external process $\bm{N}$ are independent and, in addition,
the components of $\bm{N}$ are mutually independent.
These assumptions precisely define our second-order RMF model.


\subsection{Conditional moment-generating function}\label{sec:condmgf}

A key step toward obtaining integral equations from the invariance of the Palm distribution $\mathbb{P}^0_{ij}$ is to marginalize the conservation of some stationary variables over the external stochastic drive. 
The functional form \eqref{eq:condTrans} and the key role played by MGFs in our PDE analysis suggests to consider $e^{u\lambda_i (T_{j,1}^-)}$ as that conserved quantity.
Observe that we focus on the stochastic intensity of neuron $i$, when neuron $j$ spikes, as information about the spiking neuron is erased by resets.
As the components of $\bm{N}$, the arrivals of external spikes, constitute independent Poisson point processes
with intensities $\beta_k$, we have that for all $u>0$:
\begin{eqnarray}\label{eq:defExpCons}
\hspace{20pt} \Exp{e^{u \lambda_i(T_{j,1}^-)}, \, T_{j,1}<T_{i,1}  \, \vert \, \bm{\lambda}(0)   } = 
\int_0^\infty \sum_{\bm{n}} \left(\prod_{k \neq i,j} e^{-\beta_k t} \frac{\beta_k^{n_k}}{n_k !}  \right) 
Q_{\bm{n}}(u,t)\, dt \, .
\end{eqnarray}
In the above expression, the auxiliary terms $Q_{\bm{n}}(u,t)$ are expectations over the times of spiking deliveries at fixed spiking-delivery counts
\begin{eqnarray}
\lefteqn{ Q_{\bm{n}}(u,t) = } \nonumber\\
&& \Exp{\int_0^t \!\! \! \ldots \!\! \int_0^t \prod_{k \neq i,j} \prod_{l_k=1}^{n_k}  
e^{ u\lambda_i(t \, \vert  \, \bm{N}_{t})} \rho_j(t \, \vert  \, \bm{N}_{t}) \, dt_{k,l_k} \,
\Bigg \vert \, \bm{\lambda}(0) , \bm{N}([0,t])=\bm{n}} \, ,
\end{eqnarray}
with conditional survival density functions given by
\begin{eqnarray}
\rho_j(t \, \vert  \, \bm{N}_{t}) =
 \lambda_j(t \, \vert  \, 
\bm{N}_{t}) e^{-\int_0^t (\lambda_i +\lambda_j)(s \, \vert  \,   \bm{N}_{t})  \,ds } \, .
\end{eqnarray}
In the above definitions, $\bm{N}_{t}$ is the event that $\bm{N}$ consists of the (unordered) points $t_{k,l_k}$, $1\le l_k\le n_k$ on component $k$.
Specifying the functional form of the stochastic intensities $\lambda_i(t \, \vert \,  \bm{\lambda}(0),  \bm{N})$ given in \eqref{eq:stochIntForm} and introducing the functions
\begin{eqnarray}\label{eq:auxFunc}
F_{ijk}(u,t) & = &  \nonumber\\
&& \int_0^t \exp{\left( \tau_i \mu_{ik} \left(1-\left(1-\frac{u}{\tau_i}\right)e^{-\frac{s}{\tau_i}}\right)+\tau_j \mu_{jk} \left(1-e^{-\frac{s}{\tau_j}}\right)\right)} \, ds\\
G_{ijk}(u,t) & = & \nonumber\\
&& \int_0^t e^{-s/\tau_j} \exp{\left( \tau_i \mu_{ik} \left(1-\left(1-\frac{u}{\tau_i}\right)e^{-\frac{s}{\tau_i}}\right)+\tau_j \mu_{jk} \left(1-e^{-\frac{s}{\tau_j}}\right)\right)}  \, ds,
\end{eqnarray}
we can factorize the terms $Q_{\bm{n}}(u,t)$ as 
\begin{eqnarray}
\lefteqn{ Q_{\bm{n}}(u,t) =  e^{u\tilde{\lambda}_i(t)-\int_0^t \left(\tilde{\lambda}_i(s)+\tilde{\lambda}_j(s)\right) \, ds} } \nonumber\\
&& \hspace{50pt}\prod_{k \neq i,j} F_{ijk}(u,t)^{n_k}  \left( \tilde{\lambda}_j(t) + \sum_{m \neq i,j} n_m \mu_{jm} \frac{G_{ijm}(u,t)}{F_{ijm}(u,t)} \right).
\end{eqnarray}
Utilizing the above expression in \eqref{eq:defExpCons} allows one to perform the summation over spike-delivery counts $\bm{n}$ to obtain:
\begin{eqnarray}\label{eq:condMGF}
\lefteqn{\Exp{e^{u \lambda_i(T_{j,1}^-)} , \, T_{j,1}<T_{i,1} \, \vert \, \bm{\lambda}(0)  } = } \nonumber\\
&&  \int_0^\infty e^{u\tilde{\lambda}_i(t)-\int_0^t \left(\tilde{\lambda}_i(s)+\tilde{\lambda}_j(s)\right) \, ds} e^{ \sum_{k \neq i,j} \left( F_{ijk}(u,t) -t\right)\beta_k}  \! \left( \tilde{\lambda}_j(t) + \!\!\! \sum_{m \neq i,j} \! \mu_{jm} G_{ijm}(u,t) \right) \, dt. \nonumber\\
\end{eqnarray}
The above expression can be viewed as a conditional MGF for the embedded Markov chains and capture all the necessary information to specify conservation laws about these embedded Markov chains.


\subsection{Equation for the invariant measure of the embedded chain}\label{sec:intec}

By invariance of the Palm distribution $\mathbb{P}_{ij}$ with respect to time-shifts from a spiking event of the $(i,j)$-pair to the next spiking event, we have the conservation laws
\begin{eqnarray}\label{eq:consRel}
\begin{array}{ccc}
\Expij{e^{u \lambda_i(T_{j,1}^-)} , \, T_{j,1}<T_{i,1} }  &=& \pi_j \Expj{e^{u \lambda_i}} \, \vspace{7pt} \\
 \Expij{e^{u \lambda_j(T_{i,1}^-)} , \, T_{i,1}<T_{j,1} } &=& \pi_i \Expj{e^{u \lambda_j}} \, .
\end{array}
\end{eqnarray}
In the above equalities, the right-hand term is a conditional expectation given that neuron $j$ spikes at zero,
whereas the left-hand term is a conditional expectation given that that neuron $j$ is spiking at
next spiking event of the $(i,j)$-pair.
Note that the right-hand term is closely related via equation \eqref{eq:propbInt} to the functions $h_i$ and $h_j$ featuring in the system of integral equations \eqref{eq:linFredholm} .
Moreover, the left-hand term can be related to the stationary probabilities $p_i^0$ and $p_j^0$ thanks via the conditional MGF for the embedded Markov chains \eqref{eq:condMGF}.
In fact, we have the following 

\begin{lemma}\label{lem:intSys2}
The system of equations \eqref{eq:linFredholm}  obtained in Lemma \eqref{lem:intSys2} is equivalent to the conservation laws \eqref{eq:consRel}.
\end{lemma}

\begin{proof}
Let us first express the Palm expectations of \eqref{eq:consRel} in terms of $p^0_{ij}$, the stationary measure  of $\bm{\lambda}$ under the Palm distribution $\mathbb{P}_{ij}$.
For instance, we have
\begin{eqnarray}\label{eq:condMGF2}
\Expij{e^{u \lambda_i(T_{j,1}^-)} , \, T_{j,1}<T_{i,1}   } 
&=&
\int \int \Exp{e^{u \lambda_i(T_{j,1}^-)} , \, T_{j,1}<T_{i,1} \, \vert \,\lambda_i, \lambda_j }
p^0_{ij}\big(d\lambda_i,d\lambda_j)\nonumber  \\
&=&
\pi_j \int \Exp{e^{u \lambda_i(T_{j,1}^-)} , \, T_{j,1}<T_{i,1} \, \vert \,\lambda_i+\mu_{ij}, r_j }  p^0_j\big(\lambda_i) \, d\lambda_i\nonumber  \\
& +& \pi_i \int \Exp{e^{u \lambda_i(T_{j,1}^-)} , \, T_{j,1}<T_{i,1} \, \vert \,r_i, \lambda_j+\mu_{ji} }  p^0_i\big(\lambda_j) \,  d\lambda_j\, .
\end{eqnarray}
In the first equality above, we substitute expectation with respect to the stationary distribution $\mathbb P$ for expectation with respect to the Palm distribution $\mathbb{P}_{ij}$ by virtue of the strong Markov property for $\bm{\lambda}$.
In the second equality above, we utilize the definition of the stationary measure for the post-spike embedded chain given in \eqref{eq:consDistr}, which includes reset effects.
Using expression \eqref{eq:condMGF}, we can factorize the state-dependent term of the conditional MGF appearing in \eqref{eq:condMGF2} to write
\begin{eqnarray}
\Exp{e^{u \lambda_i(T_{j,1}^-)} , \, T_{j,1}<T_{i,1} \, \vert \, \lambda_i+\mu_{ij}, r_j  }
&  = & \nonumber\\
&& \hspace{-3cm} \int_0^\infty \exp{\left( \lambda_i \left((u+\tau_i) \, e^{-\frac{t}{\tau_i}} -\tau_i \right)\right)} F_{ij}(u,t)\, dt\\
\Exp{e^{u \lambda_i(T_{j,1}^-)} , \, T_{j,1}<T_{i,1} \, \vert \, r_i, \lambda_j +\mu_{ji} } & = &  
\nonumber\\
& & \hspace{-5cm} \int_0^\infty \left(b_j+(\lambda_j+{\mu_{ji}}-b_j) e^{-\frac{t}{\tau_j}} \right)\exp{\left(  \lambda_j  \tau_j \left( e^{-\frac{t}{\tau_j}}-1\right)\right)} G_{ij}(u,t)\, dt\, ,
\end{eqnarray}
where the factors $F_{ij}(u,t)$ and $G_{ij}(u,t)$ collect the terms that are independent of the initial state $\bm{\lambda}(0)$.
Injecting these factorized representations into  \eqref{eq:condMGF2} yields
\begin{eqnarray}\label{eq:condMGF3}
& &\hspace{-1cm} \Exp{e^{u \lambda_i(T_{j,1}^-)} , \, T_{j,1}<T_{i,1}   }\nonumber \\
&=&
\pi_j  \int_0^\infty  F_{ij}(u,t)\left(\int e^{\lambda_i  \left( (u+\tau_i) \, e^{-\frac{t}{\tau_i}}-\tau_i \right)} p^0_j\big(\lambda_i) \, d\lambda_i \right)\, dt\nonumber \\
& -&  \pi_i \int_0^\infty G_{ij}(u,t) \,  \frac{\partial}{\partial t} \left( \int  e^{  \lambda_j \tau_j \left(e^{-\frac{t}{\tau_j}}-1\right)}   p^0_i\big(\lambda_j) \,  d\lambda_j \right) \, dt\nonumber \\
& +& \pi_i \int_0^\infty G_{ij}(u,t) \left(b_j+({\mu_{ji}}-b_j) e^{-\frac{t}{\tau_j}} \right) \,   \int  e^{  \lambda_j \tau_j \left(e^{-\frac{t}{\tau_j}}-1\right)}   p^0_i\big(\lambda_j) \,  d\lambda_j \, dt\, ,
\end{eqnarray}
where the initial-state dependence only appears in exponents.
We are then in a position to write \eqref{eq:condMGF2} as an equation about the functions $h_i$ and $h_j$, which are rescaled MGFs of $\lambda_i$ and $\lambda_j$ with respect to $\mathbb{P}_j$ and to $\mathbb{P}_i$, respectively.
Specifically, we have 
\begin{eqnarray}
\hspace{20pt} \Expij{e^{u \lambda_i(T_{j,1}^-)} , \, T_{j,1}<T_{i,1} } = \left(\frac{\beta_j}{\beta_i+\beta_j}\right) \Expj{e^{u \lambda_i}} = \frac{h_i\left(\tau_i \ln\left( 1+\frac{u}{\tau_i}\right)\right)}{\beta_i+\beta_j}
\end{eqnarray}
and, upon recognizing MGFs in the right-hand terms of \eqref{eq:condMGF3}, we obtain
\begin{eqnarray}
h_i\left(\tau_i \ln\left( 1+\frac{u}{\tau_i}\right)\right) & = &
\int_0^\infty  F_{ij}(u,t)  \, h_j\left( -t +\tau_i \ln\left( 1+\frac{u}{\tau_i}\right)\right) \, dt \nonumber \\
&+& \int_0^\infty G_{ij}(u,t) \,   \left[ h'_i(-t) + \left( b_j+({\mu_{ji}}-b_j)e^{-\frac{t}{\tau_j}}\right) h_i(-t) \right]
\, dt\, .
\end{eqnarray}
Adopting the short-hand notation
\begin{eqnarray}\label{eq:condMGF4}
\hspace{1.2cm} \tilde{F}_{ij}(x,y) = F_{ij}\left(\tau_i \left( e^{\frac{x}{\tau_i}}\!-\!1\right),-y\right) \; \; \mathrm{and} \; \: \tilde{G}_{ij}(x,y)=G_{ij}\left(\tau_i \left( e^{\frac{x}{\tau_i}}\!-\!1\right),-y\right) \, ,
\end{eqnarray}
we write equation \eqref{eq:condMGF4} under a form similar to that of integral equation \eqref{eq:linFredholm} by integration by parts
\begin{eqnarray}
h_i(x)
&=&
\int_{-\infty}^0  \tilde{F}_{ij}(x,y) \, h_i\left(x+y\right) \, dy \nonumber\\
&& + \int_{-\infty}^0 \tilde{G}_{ij}(x,y) \,   \left[ h'_j(y) + \left( b_j+({\mu_{ji}}-b_j)e^{\frac{y}{\tau_j}}\right) h_i(y)  \right]\, dy \nonumber\\
&=&
\int_{-\infty}^x  \tilde{F}_{ij}(x,y-x)  \, h_i\left(y\right) \, dy + \tilde{G}_{ij}(x,0)  h_j(0)\nonumber\\
&&+ \int_{-\infty}^0  \left[ -\frac{\partial}{\partial y}  \Big[ \tilde{G}_{ij}(x,y) \Big] + \left( b_j+({\mu_{ji}}-b_j)e^{\frac{y}{\tau_j}}\right)\tilde{G}_{ij}(x,y)  \right] \,   h_j(y)  \, dy \, .
\end{eqnarray}
A tedious but straightforward calculation to express the functions intervening in the above equations in terms the auxiliary functions \eqref{eq:auxFunc} yields
\begin{eqnarray}
\begin{array}{ccc}
\displaystyle \tilde{G}_{ij}(x,0) = k_i(x) \, , \quad\tilde{F}_{ij}(x,y-x) = -Q_{ij}(x,y) \, , \vspace{5pt} \\
\displaystyle \frac{\partial}{\partial y}  \Big[ \tilde{G}_{ij}(x,y) \Big]- \left( b_j+({\mu_{ji}}-b_j)e^{\frac{y}{\tau_j}}\right)\tilde{G}_{ij}(x,y)  = R_{ij}(x,y)\, .
\end{array}
\end{eqnarray}
This shows that the conservation equation \eqref{eq:consRel} is equivalent to the integral equation \eqref{eq:linFredholm} obtained via PDE analysis.
\end{proof}


\subsection{Normalization condition}\label{sec:normcond}

The previous section shows that the integral equations \eqref{eq:linFredholm}   can be interpreted as conservation laws about the embedded Markov chains of the dynamics.
There remains to find a probabilistic interpretation for the normalization condition \eqref{eq:norm} required to single out the physical solution to \eqref{eq:linFredholm}  .
It turns out that the probabilistic interpretation of \eqref{eq:norm} follows from the stationarity of the counting process $N_i+N_j$ characterizing the joint spiking activity of the $(i,j)$-pair.
Specifically, we have:

\begin{lemma}\label{lem:normCond}
The normalization condition \eqref{eq:norm} is equivalent to the Slivnyak inverse formula applied to the constant unit function with respect to the counting process $N_i+N_j$, i.e.:
\begin{eqnarray}
\Exp{1} =(\beta_i+\beta_j)\Expij{\int_0^{T_{ij,1}} 1 \, dt} = (\beta_i+\beta_j)\Expij{T_{ij,1}}
\end{eqnarray}
where $T_{ij,1} = \min{ \left( T_{i,1}, T_{j,1} \right)}$ is $(i,j)$-pair interspike. interval.
\end{lemma}

\begin{proof}
The mean time between two consecutive spiking events of the $(i,j)$-pair satisfies
\begin{eqnarray}
\Exp{T_{ij,1}  \, \vert \, \bm{\lambda}(0) } 
&=&
\ExpN{ \int_0^\infty  \Prob{ T_{ij,1}>t \, \vert \, \bm{\lambda}(0) , \bm{N}} \, dt } \, , 
\end{eqnarray}
where  $\bm{N}$ denotes the history of external spike deliveries, with components distributed as independent Poisson processes with mean intensities $\beta_k$, $k \neq i,j$.
Conditioning to a particular realization of $\bm{N}$, the next spiking event of the $(i,j)$-pair is defined as the minimum of two independent rate processes.
Moreover, the stochastic intensities $\lambda_i$ and $\lambda_j$ are the deterministic hazard rate functions associated to these two rate processes.
In particular, we have
\begin{eqnarray}
\lefteqn{ \Prob{ T_{ij,1}>t \, \vert \, \bm{\lambda}(0) , \bm{N}}
= } \nonumber\\
&& \hspace{60pt} \exp{\left(-\int_0^t \Big( \lambda_i(s\, \vert \, \bm{\lambda}(0) , \, \bm{N}) +\lambda_j(s \, \vert  \, \bm{\lambda}(0) , \, \bm{N}) \Big) \,ds \right)} \, ,
\end{eqnarray}
where we use that the overall hazard rate of the pair is the sum of individual hazard rates by conditional independence.
Exploiting the independent Poissonian nature of the external spike deliveries, a similar calculation as for the conditional MGF in Section \Cref{sec:condmgf} yields
\begin{eqnarray}\label{eq:probT}
\lefteqn{\ExpN{  \Prob{ T_{ij,1}>t \, \vert \, \bm{\lambda}(0) , \bm{N}} }
=} \nonumber\\
&& \hspace{60pt}
\int_0^\infty e^{-\int_0^t \left(\tilde{\lambda}_i(s)+\tilde{\lambda}_j(s)\right) \, ds} e^{ \sum_{k \neq i,j} \left( F_{ijk}(0,t) -t\right)\beta_k}  \, dt \, .
\end{eqnarray}
Utilizing expression \eqref{eq:consDistr} for the stationary distribution $p_{ij}^0$ of the post-spiking embedded chain, we can evaluate the expectation of the mean interspike time with respect to the Palm distribution $\mathbb{P}_{ij}$:
\begin{eqnarray}
\lefteqn{ \Expij{T_{ij,1} } 
=
\pi_j \int \Exp{T_{ij,1}  \, \vert \,\lambda_i+\mu_{ij}, r_j }  p^0_j\big(\lambda_i) \, d\lambda_i  }
\nonumber\\
&& \hspace{80pt}
+ \: \pi_i \int \Exp{T_{ij,1} \, \vert \,r_i, \lambda_j+\mu_{ji} }  p^0_i\big(\lambda_j) \,  d\lambda_j \, .
\end{eqnarray}
where the conditional expectations can be specified via formula \eqref{eq:probT}.
In fact, factorizing the initial-state dependence as in Section \Cref{sec:intec}, we have
\begin{eqnarray}
\Exp{T_{ij,1}  \, \vert \,\lambda_i+\mu_{ij}, r_j }  &=&  \int_0^\infty \exp{\left(\tau_i \lambda_i \left(e^{-\frac{t}{\tau_i}} -1 \right)\right)} H_{ij}(t) \, dt \, ,\\
\Exp{T_{ij,1} \, \vert \,r_i, \lambda_j+\mu_{ji} } &=& \int_0^\infty \exp{\left( \tau_j \lambda_j \left(e^{-\frac{t}{\tau_j}} -1 \right)\right)} H_{ji}(t) \, dt \, .
\end{eqnarray}
where the auxiliary functions $H_{ij}$ and $H_{ji}$ collect the terms that are independent of the initial state $\bm{\lambda}(0)$.
Applying the inversion formula of Palm calculus to the constant unit functions with respect to the counting process $N_i+N_j$, 
we get $(\beta_i+\beta_j)\Expij{T_{ij,1}}=1$.  
Utilizing the definitions of $\pi_i$ and $\pi_j$ in terms of the rates $\beta_i$ and $\beta_j$ given in \eqref{eq:defPi},
the inversion formula reads
\begin{eqnarray}
1
&=&
 \int_0^\infty H_{ij}(t)  \int   \beta_i \exp{\left(\tau_i \lambda_i \left(e^{-\frac{t}{\tau_i}} -1 \right)\right)} p^0_j\big(\lambda_i) \, d\lambda_i \,dt \\
&& \hspace{20pt} + \int_0^\infty H_{ji}(t)  \int  \beta_j  \exp{\left(\tau_j \lambda_j \left(e^{-\frac{t}{\tau_j}} -1 \right)\right)}  p^0_i\big(\lambda_j) \, dt \,  d\lambda_j \, , \\
&=&  \int_{-\infty}^0 H_{ij}(-x) h_i(x) \, dx +   \int_{-\infty}^0 H_{ji}(-y) h_j(y) \, dy \, ,
\end{eqnarray}
where the second equality follows from recognizing MGF functions and performing a change of variable.
Finally a tedious but straightforward calculation shows that
$H_{ij}(-x) = K_{ij}(x)$ and $H_{ji}(-x) = M_{ij}(x)$, proving the equivalence of the normalization condition \eqref{eq:norm}  to
the inversion formula applied to the constant unit function with respect to the counting process $N_i+N_j$.
\end{proof}


\section*{Acknowledgments} T.T. was supported by the Alfred P. Sloan Research Fellowship FG-2017-9554.
 F.B. was supported by an award from the Simons Foundation (\#197982). Both awards are
to the University of Texas at Austin.

\bibliographystyle{siamplain}
\bibliography{siam_bib}

\end{document}